\documentclass[a4paper]{article}

\usepackage{latexsym,amsthm,amscd}
\usepackage[all]{xy}
\usepackage{rotating}
\usepackage[utf8]{inputenc}
\usepackage[T1]{fontenc}
\usepackage[english]{babel}
\usepackage{amsmath}
\usepackage{graphicx}
\usepackage{amssymb}
\usepackage{amsfonts}
\usepackage[svgnames]{xcolor}
\usepackage{geometry}
\usepackage{fancyhdr}
\usepackage{lastpage} 
\usepackage{array}
\usepackage{stmaryrd}
\usepackage{hyperref}
\usepackage{tikz}
\usepackage{bm}
\usepackage[french]{minitoc}
\usepackage{silence}
\usepackage{enumitem}
\usepackage{float}

\usetikzlibrary{patterns}
\usepgflibrary{decorations.pathmorphing}
\usetikzlibrary{decorations.pathmorphing}
\usetikzlibrary{shadows}
\usetikzlibrary{positioning}
\usetikzlibrary{arrows,automata}

\newtheorem*{rep@theorem}{\rep@title}
\newcommand{\newreptheorem}[2]{%
\newenvironment{rep#1}[1]{%
 \def\rep@title{#2 \ref{##1}}%
 \begin{rep@theorem}}%
 {\end{rep@theorem}}}

\newtheorem{theorem}{Theorem}
\newtheorem{proposition}[theorem]{Proposition}
\newtheorem{definition}[theorem]{Definition}

\newtheorem{corollary}[theorem]{Corollary}
\newtheorem{lemma}[theorem]{Lemma}
\newtheorem{example}[theorem]{Example}

\newtheorem{remark}{Remark}

\newreptheorem{theorem}{Theorem}
\newreptheorem{lemma}{Lemma}

\title{On the computability properties of topological entropy: a general approach}

\author{Silvère Gangloff, Alonso Herrera, Cristobal Rojas and Mathieu Sablik}

\def\N{\mathbb N}

\def\Z{\mathbb Z}
\def\Q{\mathbb Q}

\renewcommand{\vec}[1]{\textbf{#1}}

\newcommand{\define}[1]{\textbf{#1}}

\newcommand{\cbra}[1]{\left\{ #1 \right\}}

\begin{document}

\maketitle

\begin{abstract}
The dynamics of symbolic systems, such as multidimensional subshifts of finite type or cellular automata, 
are known to be closely related to computability theory. In particular, the appropriate tools to describe and classify topological entropy for this kind of systems turned out to be of computational nature. Part of the great importance of these symbolic systems relies on the role they have played in understanding more general systems over non-symbolic spaces. The aim of this article is to investigate topological entropy from a computability point of view in this more general, not necessarily symbolic setting. In analogy to effective subshifts, we consider computable maps over effective compact sets in general metric spaces, and study the computability properties of their topological entropies. We show that even in this general setting, the entropy is always a $\Sigma_2$-computable number. We then study how various dynamical and analytical constrains affect this upper bound, and prove that it can be  lowered in different ways depending on the constraint considered. In particular, we obtain that all $\Sigma_2$-computable numbers can already be realized within the class of surjective computable maps over $\{0,1\}^{\N}$, but that this bound decreases to $\Pi_{1}$(or upper)-computable numbers when restricted to expansive maps.  On the other hand, if we change the geometry of the ambient space from the symbolic  $\{0,1\}^{\N}$ to the unit interval $[0,1]$, then we find a quite different situation --  we show that the possible entropies of computable systems over $[0,1]$ are exactly the $\Sigma_{1}$(or lower)-computable numbers and that this characterization switches down to precisely the computable numbers when we restrict the class of system to the quadratic family.  
\end{abstract}

\tableofcontents


\section{Introduction}

Introduced by Adler, Konheim and McAndrew in the 60's \cite{Adler-Konheim-McAndrew-1965}, topological entropy has become (one of) the most important topological invariants in dynamical systems theory. It expresses   the exponential growth rate of the number of distinguishable orbits the system can create under iteration. Topological entropy has many important applications, specially as a measure of the degree of disorder of a system. It is for instance widely used in the theoretical description of the transition to chaos. As pointed out by Milnor \cite{Mil02}, it is natural to wonder whether topological entropy can be \textit{effectively computed}: 

\bigskip
\noindent \textit{given an explicit (finitely described) dynamical system and given $\epsilon>0$, is it possible to compute the associated topological  entropy with a maximum error of $\epsilon$? }
\medskip

In the most general case the answer to this problem is known to be negative. For example, it was proven in \cite{Hurd-Kari-Culik-1992} that for the class of cellular automata, there is no algorithm to compute (or even approximate) the topological entropy, and that even the problem of determining whether a cellular automaton has zero entropy, is undecidable. Similar results have been obtained, for instance, for the class of iterated piecewise-affine maps~\cite{Koiran-2001} in compact spaces of dimension $d\geq 4$.  

On the other hand, it is quite clear that the entropy of a (finitely described) dynamical system is always in the arithmetical hierarchy of real numbers studied in~\cite{Zheng-Weihrauch-2001}, and this leads to the formulation of the following:

\bigskip
\noindent\textbf{Problem 1}:  \textit{classify classes of dynamical systems according to the exact position their entropies have in the arithmetical hierarchy. }
\medskip

Let us briefly recall this hierarchy. A real number $h$ is $\Sigma_1$ (or \textit{lower-computable}) if it is the supremum of a computable sequence of rationals. Simetrically, $h$ is $\Pi_1$ (or \textit{upper-computable}) if $-h$ is $\Sigma_1$. Higher positions in this hierarchy are defined inductively:  $\Sigma_n$ numbers are suprema of computable sequences of $\Pi_{n-1}$ numbers and $\Pi_{n}$ numbers are infima of computable sequences of $\Sigma_{n-1}$ numbers. The lowest position in this hierarchy is occupied by the \textit{computable} real numbers: those being simultaneously lower and upper computable.

The arithmetical hierarchy is thus a classification of real numbers according to their effective descriptive complexity, which is in turn related to their degrees of unsolvability \cite{unsovability}. Characterizing entropies within this hierarchy can thus be seen as shedding some light into the computational capabilities of the class of dynamical systems considered. A form of this question for the class of cellular automata and the class of piecewise-affine maps was first stated in \cite{Koiran-2001}.  

In a major breakthrough \cite{Hochman-2009}, a solution to Problem 1 was given for the class of cellular automata in dimension $d\geq 3$: their possible entropies are exactly the $\Sigma_2$-computable real numbers. It was later shown \cite{Guillon-Zinoviadis-2012} that the same characterization holds for cellular automata in dimension $d=2$, but that when $d=1$ the complexity of entropies for this class decreases to exactly the  $\Pi_1$-computable positive reals.  

In this article we propose to study Problem 1 in a more systematic way by considering general \textit{computable} dynamical systems as presented in~\cite{Galatolo-Hoyrup-Rojas-2011}.  A computable dynamical system can be thought of as one for which one has a computer program that can simulate the system in the sense that individual trajectories can be computed to any desired degree of accuracy. The associated computer program can then be thought of as a finite description of the system, and therefore the question of algorithmic computability of the associated entropy makes sense. In this setting, we solve Problem 1 for different classes of dynamical systems. Our main results are the following.

\medskip

\noindent\textbf{Theorem A}. \textit{The topological entropy of a general computable dynamical system is always a $\Sigma_2$-computable real number.  
}

\medskip

Thus, we have a definite upper bound applicable to all computable dynamical systems over general metric spaces. Note that in particular, Theorem A tells us that the class of 2D cellular automata is \emph{entropy-complete} in the sense that their entropies get as complex as they can be, exhibiting the full spectrum of all $\Sigma_{2}$-computable numbers. On the other hand, we also know that 1D cellular automata are not complete in this sense, and it is natural to ask whether the more general class of 1D computable dynamical systems can be complete. We answer this question in the affirmative, even within the class of systems that are constrained to be surjective.  

\medskip

\noindent\textbf{Theorem B}.  \textit{Every positive $\Sigma_{2}$-computable real number is the topological entropy of a surjective computable map over $\{0,1\}^{\N}$.} 
\medskip

We then analyse the effect of changing the geometry of the ambient space and study Problem 1 for computable systems over $[0,1]$. As shown by the next result, the situation in this space is very different: in particular, not even a single $\Sigma_{2}$-computable number can be realised within this class. 
 
 \medskip
 
 \noindent\textbf{Theorem C}. \textit{The possible entropies of computable systems over $[0,1]$ are exactly the $\Sigma_{1}$-computable numbers.}

\medskip
 
 Finally, we consider the seemingly very restricted class of quadratic maps. More precisely, we consider the so called logistic family given by $f_{r}(x) = rx(1-x)$, with $r\in [0,4]$. This is a very well studied family with an enormous amount of available theory. We apply some of the available tools to show that the possible entropies for this family are exactly the computable numbers. 

\medskip

 \noindent\textbf{Theorem D}. \textit{A real number $h\in[0,\log(2)]$ is computable if and only if it is the entropy of a computable logistic map}.  

\medskip

The present work can be seen as part of a recent research trend in which dynamical systems are studied from a computational complexity point of view \cite{BY, BBRY, BRS, Lorenz,  Kawa, DudYam2, Aubrun-Sablik-2010, Moore1, Kurka, burr_schmoll_wolf, SpectralGap}.  The general idea is to understand how the computability properties of the main invariants that describe a given system are related to its dynamical, geometrical and analytical properties. In particular, a major goal is to characterize the properties that make these invariants \emph{easier} to compute.  The techniques developed for these results usually involve delicate constructions allowing to see the dynamical system as a ``computing device'', which can then be ``programmed'' to exhibit some particular behavior. Our Theorems B and C will implement such constructions.

After some preliminaries on computable analysis, we introduce computable dynamical systems in Section \ref{sec.comp.general}. Section \ref{sec.growth.type.metric} is devoted to the proof of Theorem A. In Section \ref{section.symbolic.dynamics} we study Problem 1 for systems on Cantor sets, in particular proving Theorem B.  Finally, the proofs of Theorems C and D are presented in Section \ref{section.continuous.dynamics}.


\section{ \label{sec.comp.general} Preliminaries}

\subsection{A bit of computable analysis.}

In this section, we recall notions from
computable analysis on metric spaces. 
All the results presented here are well known. 
For a reference on computable analysis, 
see~\cite{Brattka}.  In the following we will make use of the word \textit{algorithm} to mean a computer program written in any standard programming language or, more formally, a Turing Machine. Algorithms are assumed to be only capable of manipulating integers. By identifying countable sets  with integers in a constructive way, we can let algorithms work on these countable sets as well.  For example, algorithms can manipulate rational numbers by identifying each $p/q$ with some integer $n$ in such a way that both $p$ and $q$ can be computed from $n$, and vice-versa. We fix such a numbering from now on.

\subsubsection{Computable metric spaces.}

\begin{definition}
A \textbf{computable metric space} is a triple $(X,d,\mathcal{S})$, where $(X,d)$ is a metric space 
and $\mathcal{S} = \{s_i : i \ge 0\}$ a countable dense subset of $X$, 
whose elements are called \textbf{ideal points}, 
such that there exists an algorithm 
which, upon input $(i,j,n) \in \N^{3}$, 
outputs  $r\in \Q$ such that 
$$|d(s_i,s_j)-r| \le 2^{-n}.$$
We say that the distances between ideal points 
are uniformly computable. 
\end{definition}

For $r>0$ a 
rational number and $x$ an element of $X$,
we denote $B(x,r) = \{ z \in X \ : \ d(z,x)<r\}$, 
the ball with center $x$ and radius $r$.
The balls centered on elements of 
$\mathcal{S}$ with rational radii are called \textbf{ideal balls}. A computable enumeration of the ideal balls $B_{n}=B(s^{(n)},r^{(n)})$ can be obtained by taking for instance a bi-computable bijection $\varphi:\N \rightarrow 
\N \times \Q$ and letting $s^{(n)} = s_{\varphi_1 (n)}$ and $r^{(n)} = \varphi_2 (n)$, where $\varphi(n) = (\varphi_1 (n),\varphi_2 (n))$.  We fix such a computable enumeration from now on. For any 
subset $I$ (finite or infinite) 
of $\N$, we denote $\mathcal{U}_{I}$
the collection of ideal balls $B_n$ with $n \in I$, and 
$U_{I}$ the union of these balls: 
\[U_{I} = \bigcup_{n \in I} B_n.\]

\begin{example}The following are two important examples. 
\begin{itemize}
\item For a finite alphabet  $\mathcal{A}$   
and $\#$ one of its elements, the Cantor space $\mathcal{A}^{\N}$ with its usual 
metric has a natural computable metric space structure where the ideal points can be taken to be $\mathcal{S} = \{w\#^{\infty}: 
w \in \mathcal{A}^{*}\}$. In this case, the ideal balls are 
the cylinders. In the same way, spaces like $\mathcal{A}^{\Z^{d}}$ with $ d \ge 1$ can also be endowed with a natural computable metric space structure. 
\item The compact interval $[0,1]$, with its usual metric 
and $\mathcal{S}=\Q\cap [0,1]$ is also  
a computable metric space. The ideal balls here are the 
open intervals with rational endpoints. 
\end{itemize}
\end{example}

\begin{definition}
Let $(X,d,\mathcal{S})$ be a computable metric space
such that $(X,d)$ is compact.
$X$ is said to be \textbf{recursively compact} 
if the inclusion $$X \subset U_I,$$
where $I$ is some finite subset of $\N$, is semi-decidable. 
This means that there is an algorithm which, given  $I$ as input, halts if and only if 
the inclusion above is verified. 
\end{definition}

\begin{example}
The Cantor space and the compact interval 
are easily seen to be recursively compact.
\end{example}

\subsubsection{Computable closed subsets}

\begin{definition}
Let $(X,d,\mathcal{S})$ be a computable metric space.
A closed subset $K \subset X$ 
is said to be \textbf{effective} when there exists 
an algorithm
such that $$X \backslash K = U_I,$$
where $I$ is the set of integers 
on which the algorithm halts.
\end{definition}

Effective closed sets are also known as closed co-r.e. sets or upper computable. 
We will also make use of the following fact.

\begin{proposition}
\label{prop.equivalence.comp.closed.subset}
Let $(X,d,\mathcal{S})$ be a recursively compact
computable metric space.
A closed subset $K \subset X$ is  effective if and 
only if the inclusion
$$K \subset U_I \qquad \text{ where } I \subset \N \text{ is finite }$$ 
is semi-decidable.
\end{proposition}

\begin{proposition}
Let $(X,d,\mathcal{S})$ 
be a recursively 
compact metric space. 
The inclusion $\overline{B_n} 
\subset B_m$ is semi-decidable.
\end{proposition}

\begin{proof}
Note that $\overline{B_n}$ is effective and apply Proposition~\ref{prop.equivalence.comp.closed.subset}.
\end{proof}

\subsubsection{Computable functions}

\begin{definition}
Let $(X,d)$ and $(X',d')$ be computable metric spaces. Let us denote $(B'_m)_m$ 
an enumeration of ideal balls of $X'$.
A function $f : X \rightarrow X'$ is 
\textbf{computable} if there exists an algorithm 
which, given as input some integer $m$, 
enumerates a set $I_m$ 
such that $$f^{-1} (B'_m) = U_{I_m}.$$
\end{definition}

It follows that computable functions are continuous. It is perhaps more intuitively familiar, and provably equivalent, to think of a computable function as one for which there is an algorithm which, provided with arbitrarily good approximations of $x$, outputs arbitrarily good approximations of $f(x)$.  In symbolic spaces, this can easily be made precise:

\begin{example}
In Cantor space $\mathcal{A}^{\N}$, a function 
$f : \mathcal{A}^{\N} \rightarrow \mathcal{A}^{\N}$ is computable if and only if there exists some non decreasing 
computable function $\varphi : \N \rightarrow \N$ together with an algorithm which, provided with the $\varphi(n)$ first symbols of the sequence $x$, computes the $n$th first symbols of the sequence $f(x)$. 
\end{example}

For $X$ a compact computable
metric space, we consider the product space $X\times X$ with the 
distance $d^2$ defined for all $(x,y), 
(x',y') \in X 
\times X$ by 
\[d^2 ((x,y),(x',y')) = \max (d(x,y),d(x',y')).\]
The space $(X ^2 , d^2 , \mathcal{S}^2)$ 
is a compact computable metric space.
The ideal balls of this space are the 
products of ideal balls of $X$.

\subsection{Computable dynamical systems and topological entropy}

Throughout this section $(X,d)$ will denote a compact metric space. We will assume a computable structure on $X$ whenever necessary.  

\begin{definition}\label{computable-dynamics}
A \textbf{dynamical system} on $X$ is a pair $(K,f)$, where $f : X \rightarrow X$ is
a continuous function and $K$ 
is a compact subset of $X$ such that $f(K)\subset K$.  A dynamical system is \textbf{computable} whenever $f$ is computable and $K$ is effective. 
\end{definition}

We very briefly recall the definition of topological entropy and state some of its elementary 
properties. We refer the reader to \cite{Adler-Konheim-McAndrew-1965} for more details. 
Let $(X,d)$ be a metric space and $K$ be a compact subset of $X$. A \textbf{cover} of $K$ is a 
collection $\mathcal{U}$ of open subsets of $X$ such  that 
\[K \subset \bigcup_{U \in \mathcal{U}} U.\] 

A \define{sub-cover} of a
cover $\mathcal{U}$ is a subset of 
this set which is also a cover. Let $\mathcal{U}$ and $\mathcal{U}'$ be two 
collections of open subsets of $X$. 
We say that $\mathcal{U}$ is 
\define{thinner} than $\mathcal{U}'$ when for all 
$U \in \mathcal{U}$, there is some $V \in \mathcal{U}'$ 
such that $U \subset V$. We denote this relation $\mathcal{U}'\prec \mathcal{U}$.
For a pair of open covers $\mathcal{U}$ and 
$\mathcal{U}'$, their \define{join}  $\mathcal{U}\vee\mathcal{U'}$ is defined by 
\[\mathcal{U}\vee\mathcal{U}'=\{U\cap U' , U \in \mathcal{U}, \ U' \in \mathcal{U}'\}.\]

We remark that for a continuous function $X \rightarrow X$,
$K$ a compact subset of $X$ and an 
open cover $\mathcal{U}$ of $K$, the set 
\[f^{-1} (\mathcal{U}) : = \{ f^{-1} (U) \ : \ U \in \mathcal{U}\}\] is an open cover.

\begin{definition}
The \textbf{entropy} of a 
dynamical system $(K,f)$ of $X$
\define{relative to an open cover} 
$\mathcal{U}$ is defined by
\[h(K,f,\mathcal{U})=\inf_n\frac{\log_2 
\left(N_n\left(K,f,\mathcal{U}\right)\right)}{n}\]
where $N_n\left(K,f,\mathcal{U}\right)$ is 
the minimal cardinality 
of a sub-cover of \[\bigvee_{k \le n-1} f^{-k} (\mathcal{U}).\]
\end{definition}

\begin{remark}
It is a well known fact that 
this infimum 
is also a limit. This fact relies 
on the sub-additivity of 
the sequence $\left(N_n\left(K,f,
\mathcal{U}\right)\right)_{n\in\N}$. 
\end{remark}

\begin{lemma} 
If $\mathcal{U}$ and $\mathcal{U}'$ are two open covers 
such that $\mathcal{U} \prec \mathcal{U}'$, then 
$h(K,f,\mathcal{U}) \ge h(K,f,\mathcal{U}')$.
\end{lemma}

\begin{definition}
The \define{topological entropy} of a dynamical 
system $(K,f)$ 
is defined as
\[h(K,f)=\sup\left\{h(K,f,\mathcal{U}) \ 
| \ \mathcal{U} \textrm{ 
finite open cover of }K\right\}.\]
\end{definition}

When the set $K$ is clear from the context, we will simply write $h(f)$. We recall that two dynamical systems $(K,f)$ and $(K',g)$ are said to be conjugated whenever there exists a homeomorphism $\phi:K\to K'$ such that $g\circ \phi = \phi \circ f $ over $K$.  It is well known that conjugated systems have the same topological entropy. 

\begin{lemma}
Let $(K,f)$ be a dynamical 
system in $X$. Then, for all $n \ge 1$, the equality $h(f^n)=nh(f)$ holds.
\end{lemma}


\section{\label{sec.growth.type.metric}General obstructions and entropy-complete systems.}


Let us recall that a real number $x$ is $\Sigma_{2}$-computable if there is an algorithm that computes a double sequence $(q_{ij})$ of rationals satisfying $$x=\sup_{i}\inf_{j} q_{ij}.$$
The aim of this section is to prove 
the following theorem: 

\begin{theorem}\label{theorem.obstruction}
Let $(X,d,\mathcal{S})$ be a recursively compact metric 
space and $(K,f)$ be a computable 
dynamical system in $X$. Then the topological entropy $h(K,f)$ is a $\Sigma_2$-computable number.
\end{theorem}

The proof will follow from the following intermediate statement: 

\begin{proposition}
\label{prop.computing.number.covers}
Let $(X,d,\mathcal{S})$ be a recursively compact metric 
space and $(K,f)$  a computable dynamical system 
in $X$.  Then there is an algorithm which, given a finite $I\subset \N$ and $n\in \N$, outputs a sequence of rational numbers $h_{i}$ such that  
\[\left(N_n\left(K,f,
\mathcal{U}_I\right)\right)_{I,n}= \inf_{i}h_{i}. \]

\end{proposition}

\begin{proof} Let $(K,f)$ be a computable dynamical system in $X$. A consequence 
of the computability of $f$ is that there 
exists some algorithm $\mathcal{A}$ 
which on input $k,m$ enumerates a set $I_{(k,m)}$
such that $f^{-k} (B_m) = U_{I_{(k,m)}}$.
Let us consider the algorithm which given as input $(I,n,i)$, where $I$ is a finite set and $n,i$ are integers, performs the following steps:  

\begin{enumerate}
\item initialization of a variable $c$ 
with the value $c = |I|^n$.
\item test if $\mathcal{U}_I$
is a cover of $K$ (using the recursive compacity 
of $X$): if not, then set $c=0$.
\item for all sets $J \subset \llbracket 0,i \rrbracket$ and for all 
$L$ subset of 
$\llbracket 1, |I|\rrbracket ^n$, it 
does the following operations: 
\begin{enumerate}
\item test if $\mathcal{U}_J$
is a cover of $K$ (using the recursive compactness 
of $X$).
\item if true, for all $k \in J$, test if 
there exists some element $(l_1, \dots, l_{n})$ 
of $L$ such that
\[\overline{B_k} \subset 
\bigcap_{i \le n-1} U_{\llbracket 0,i \rrbracket 
\cap I_{(i,m[l_i])}},\]
where $m[l_i]$ is the $l_i$th 
element of $I$. The set $U_{\llbracket 0,i \rrbracket 
\cap I_{(i,m[l_i])}}$ can be considered as 
an approximation of $f^{-i} (B_{m[l_i]})$. 
If the test is true for all $k$, 
this implies that $\mathcal{U}_J$ 
is a thinner cover than 
\[\left\{ \bigcap_{i \le n-1} 
f^{-i} (B_{m[l_i]}) : (l_1,...,l_n) \in L\right\} \subset \bigvee_{k \le n-1} f^{-k} (\mathcal{U}_I),\]
and in particular that this set is a subcover 
of $\bigvee_{k \le n-1} f^{-k} (\mathcal{U}_I)$.
In this case, attribute the value $\min(|L|,c)$
to $c$.
\end{enumerate}   
\item output the value of $c$.          
\end{enumerate}

Denote $\xi (I,n,i)$ the output of this 
algorithm.
We have, by definition, the equality 
\[\left(N_n\left(K,f,
\mathcal{U}_I\right)\right)_{I,n} 
= \inf_{i} \xi (I,n,i).\]

\end{proof}

\begin{proof}[Proof of Theorem~\ref{theorem.obstruction}:]
Let $(K,f)$ be a dynamical system in $X$. 
The supremum in the definition of 
the topological entropy can be taken 
over finite open covers of $K$ with 
ideal balls. Indeed, for any cover $\mathcal{U}$
of $K$, any $U \in \mathcal{U}$ 
and any $x \in U$, 
there exists an ideal ball 
$B_{n(x)} \subset U$ that contains $x$. 
The cover $\{B_{n(x)} , x \in K\}$ 
clearly admits as a sub-cover some $\mathcal{U}_I$ with finite $I\subset \N$, which is therefore thinner than 
$\mathcal{U}$. Thus, 
\[h(K,f,\mathcal{U})\le h(K,f,\mathcal{U}_I).\] 
As a consequence, 
\[h(K,f) = \sup_{I} h(K,f,\mathcal{U}_I),\]
where the supremum is taken now over all the finite 
set of integers. 
The result now follows from Proposition~\ref{prop.computing.number.covers}, 
and the fact that the logarithm 
is a computable function. 
\end{proof}

\begin{remark}We note that the given proof of Theorem \ref{theorem.obstruction}  \emph{relativizes} to $(K,f)$ in the sense that, when $(K,f)$ is not computable, we still have an algorithm that can correctly $\Sigma_{2}$-compute the topological entropy if provided with arbitrarily good  descriptions of $K$ as an effective set and of $f$ as a computable function.\end{remark}

Theorem \ref{theorem.obstruction} gives us a definite upper bound on how computationally hard the entropy of a system can be.  This motivates the following definition:

\begin{definition} A class of computable dynamical systems is said to be \emph{entropy-complete} if all non-negative $\Sigma_{2}$-computable reals can be realized as the entropy of a system in the class. 
\end{definition}

As mentioned in the introduction, it is known that the class of Cellular Automata in dimension $d=2$ is entropy-complete, but not in dimension $d=1$.   It turns out that classes of systems admitting generating covers cannot be entropy-complete either.   

\begin{definition}
Let $(X,d)$ be a compact metric space,
$f : X \rightarrow X$ a continuous function, 
and $\mathcal{U}$ a finite open cover of 
$X$. We say that $\mathcal{U}$ is a generating 
cover for $X$ when for all $\mathcal{U}'$ finite 
open cover of $X$, there exists some 
$n$ such that the cover 
\[\bigvee_{k \le n-1} f^{-k} (\mathcal{U})\]
is thinner than $\mathcal{U}'$.
\end{definition}

\begin{proposition}
Let $(X,d,\mathcal{S})$ be a recursively 
compact metric space, and $(K,f)$ a 
computable dynamical system in $X$. If 
$(K,f)$ has a generating cover, then its entropy 
is a $\Pi_1$-computable number. 
Its upper (resp. lower) 
entropy dimension is a $\Sigma_2$ (resp. 
$\Pi_2$) computable 
number.
\end{proposition}
\begin{proof}
Note that if $\mathcal{U}$ is a generating cover, then 
 \[h(K,f) =  h(K,f,\mathcal{U}),\]
 and the results follows from Proposition~\ref{prop.computing.number.covers}.
\end{proof}

An interesting application of the previous proposition is the following.  Recall that a dynamical 
system is called \emph{expansive} when there exists $\alpha>0$ such that 
for any ordered pair of points $(x,y) \in X^2$, 
there exists some $n$ such that 
\[d(f^n(x), f^n(y)) \ge \alpha.\] 

It is a well known fact that expansive dynamical systems admit generating covers \cite{Kurka}, which leads us to the following:

\begin{corollary}\label{expansive} If a computable system $(K,f)$ is expansive, then its topological entropy is a $\Pi_1$-computable number. 
\end{corollary}


\section{Computable systems on Cantor sets}\label{section.symbolic.dynamics}

In this section we 
consider dynamical systems over Cantor sets, that is, over the metric space $\mathcal{A}^{\N}$  for some finite alphabet $\mathcal{A}$. The purpose is to show that computable systems over these spaces are \emph{entropy-complete} in the sense that their entropies  exhibit the full spectrum of all $\Sigma_{2}$-computable numbers, even when restricting the systems to be surjective.  
We start in 
Section~\ref{sec.characterization.entropy.cantor.set} by first providing a construction that yields the desired result for general computable systems.  In Section~\ref{sec.surjectivity} we 
ameliorate the result by means of a different construction that yields the result for surjective systems which, moreover, act on the canonical Cantor space $\{0,1\}^{\N}$.  Although the second result implies the first one, the ideas involved are different and, we believe, of independent interest. 

\subsection{\label{sec.characterization.entropy.cantor.set}
A characterization for general systems}

In the following, a \emph{pattern} on a finite set $\mathcal{A}$ is an element of $\mathcal{A}^{\mathbb{U}}$, with $\mathbb{U} \subset \N \times \Z$ a finite set. We say that a pattern $p \in \mathcal{A}^{\mathbb{U}}$ appears 
in $x \in \mathcal{A}^{\N \times \Z}$ when 
there exists some $(i,j) \in \N \times \Z$ 
such that $x_{|(i,j)+\mathbb{U}}=p$. 
A compact subset $K$ of $\mathcal{A}^{\N \times \Z}$ is effective when 
there is an algorithm which 
enumerates a set of patterns 
such that $K$ is the set of configurations 
in which none of these patterns appear. \bigskip

Let us denote by $\sigma^{\vec{e}^2}:\{0,1\}^{\mathbb{N} \times \mathbb{Z}} \rightarrow \{0,1\}^{\mathbb{N} \times \mathbb{Z}}$ the function given by 
$$
\sigma^{\vec{e}^2} (x) = (x_{i,j+1})  \qquad \text{ where } \quad x = (x_{i,j}) \in \{0,1\}^{\mathbb{N} \times \mathbb{Z}}.
$$ 

That is, the superscript $\vec{e}^2$ corresponds to the 
restriction to the vertical direction of the 
two-dimensional shift action. For all $m \ge 0$, 
we denote $\Delta_m = \llbracket 0,m \rrbracket \times \llbracket -m,m \rrbracket$, where $\llbracket i,j \rrbracket$ denotes the set of integers $\{i,i+1,\dots,j\}$.
$\Delta_{-1}$ will denote the empty set. See Figure~\ref{fig.delta} 
for an illustration. 

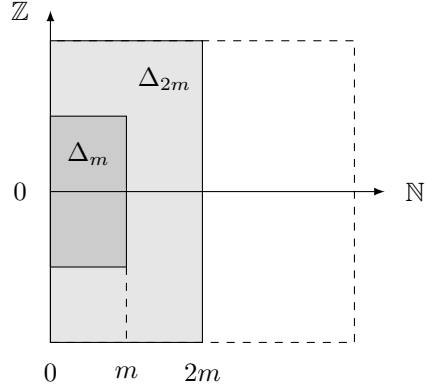
\begin{figure}[ht]
\[\begin{tikzpicture}[scale=0.2]
\fill[gray!20] (0,-10) rectangle (10,10);
\draw (0,-10) rectangle (10,10);
\fill[gray!40] (0,-5) rectangle (5,5);
\draw (0,-5) rectangle (5,5);
\draw[-latex] (0,-10) -- (0,12); 
\draw[-latex] (0,0) -- (22,0); 
\node at (-2,12) {$\Z$};
\node at (24,0) {$\N$};
\node at (-2,0) {$0$};
\node at (0,-12) {$0$};
\draw[dashed] (0,-10) -- (20,-10) -- (20,10) -- (0,10);
\draw[dashed] (5,-5) -- (5,-10); 
\node at (5,-12) {$m$};
\node at (10,-12) {$2m$};
\node at (2.5,2.5) {$\Delta_m$};
\node at (7.5,7.5) {$\Delta_{2m}$};
\end{tikzpicture}\]
\caption{\label{fig.delta} Illustration of the definition of 
the sets $\Delta_m$ and $\Delta_{2m}$ for some $m \ge 0$.}
\end{figure}

For all $l \ge 0$, 
we denote $\mathcal{U}_l$ the 
open cover of $\{0,1,\#\}^{\N \times \Z}$ 
whose elements are the cylinders associated to $\Delta_l$.
The sequence $(\mathcal{U}_l)$ is a \textit{generating 
sequence of open covers}, meaning that 
for any finite open cover $\mathcal{U}$ of 
$\{0,1,\#\}^{\N \times \Z}$, there exists 
some $l$ such that $\mathcal{U}_l$ is thinner 
than $\mathcal{U}$. Our first result, stated as Proposition~\ref{prop.cantor.real}, will make use of the 
following well known theorem:

\begin{theorem}[\cite{Hochman-2009}]
\label{thm.realization.eds}
A non-negative real number is
entropy of a dynamical 
system $(K,\sigma^{\vec{e}^2})$, 
where $K$ is an effective closed 
subset of $\{0,1\}^{\N \times \Z}$, if 
and only if it is $\Sigma_2$-computable.
\end{theorem}

We now present the announced construction. The idea is to complete the dynamical 
system $(K,\sigma^{\vec{e}^2})$, where $K \subset \{0,1\}^{\N \times \Z}$ is effective, into 
a computable dynamical system $(\{0,1,\#\}^{\N \times \Z},f)$. The function $f$ behaves like 
$\sigma^{\vec{e}^2}$, except when it 
detects a pattern which is forbidden in $K$ (an ``error''). 
In this case, it introduces a special symbol $\#$. 
This symbol is propagated by the dynamics to the whole 
half plane. In other words, the dynamics is such that the set $K$ acts as a ``repeller'' in the sense that everything that is not in $K$ ends up converging to the sequence $\#\#\#\#\#\dots$. Thus, the only part of the dynamics supporting entropy is the set $K$ itself on which $f$ acts just like $\sigma^{\vec{e}^2}$. In particular, they must have the same entropy. We then simply recode the obtained map into a computable dynamical system $(g,\{0,1,\#\}^{\mathbb{N}})$.

\begin{proposition} \label{prop.cantor.real}
A non-negative real number 
is the entropy of a computable dynamical system $(\mathcal{A}^{\mathbb{N}},f)$ if and only if it is $\Sigma_2$-computable.
\end{proposition}

\begin{proof}The fact that the entropy of 
a computable dynamical system $(\mathcal{A}^{\mathbb{N}},f)$ 
is $\Sigma_2$-computable 
derives from Theorem~\ref{theorem.obstruction}. We prove the other direction of the equivalence. 
Let $h$ be a non-negative 
$\Sigma_2$-computable real number 
and $(K,\sigma^{\textbf{e}^2})$ an effective 
dynamical system whose entropy is $h$. 
Since $K$ is an effective subset 
of $\{0,1\}^{\N \times \Z}$, there exists 
an algorithm $\mathcal{A}$ 
which enumerates a set of 
patterns $\mathcal{F}$ such that $K$ 
is defined as the set of elements of 
$\{0,1\}^{\N \times \Z}$ in which no pattern 
of $\mathcal{F}$ appear. We denote $\mathcal{F}=\{p_1, p_2, ... \}$, where the subscript $i$ in 
the notation $p_i$ corresponds to the 
order of enumeration by $\mathcal{A}$.

\begin{enumerate}
\item \textbf{Completing $(K,\sigma^{\textbf{e}^2})$ using error detection:}

Let us construct a function $f : 
\{0,1,\#\}^{\N \times \Z} \rightarrow \{0,1,\#\}^{\N \times \Z}$ whose restriction on $K$ is $\sigma^{\textbf{e}^2}$. Let us fix some $x \in \{0,1,\#\}^{\N \times \Z}$. We define $f(x)$ 
by defining successively its restrictions on all 
the sets $\Delta_{m+1} \setminus 
\Delta_m$, $m \ge -1$, as follows: 

\begin{enumerate}
\item If one of the following conditions 
is verified, then for all $\textbf{v} 
\in \Delta_{m+1} \backslash \Delta_m$, 
$f(x)_{\textbf{v}} = \#$:
\begin{itemize}
\item $ \exists \textbf{u} \in \Delta_{2(m+1)}$ such that $x_{\textbf{u}}=\#$ 
\item the restriction of 
$x$ to some subset of $\Delta_{2(m+1)}$ is 
in $\{p_1,...,p_{m+1}\}$.
\end{itemize} 
\item else, for all $\vec{v}$ in $\Delta_{m+1} \backslash \Delta_m$, 
$f(x)_{\vec{v}} = x_{\vec{v}+\vec{e}^2}$.
\end{enumerate}

The symbol $\#$ is used as an \textit{error symbol}. Indeed, 
when a forbidden pattern is detected, the function writes 
a symbol $\#$ which then propagates on $\N \times 
\Z$: when $x$ contains some symbol $\#$ on some $\Delta_{m}$, 
then on every position outside of $\Delta_{\lceil m/2 \rceil}$, the symbol of $f(x)$ on this position
is $\#$. See an illustration on Figure~\ref{fig.diez}.

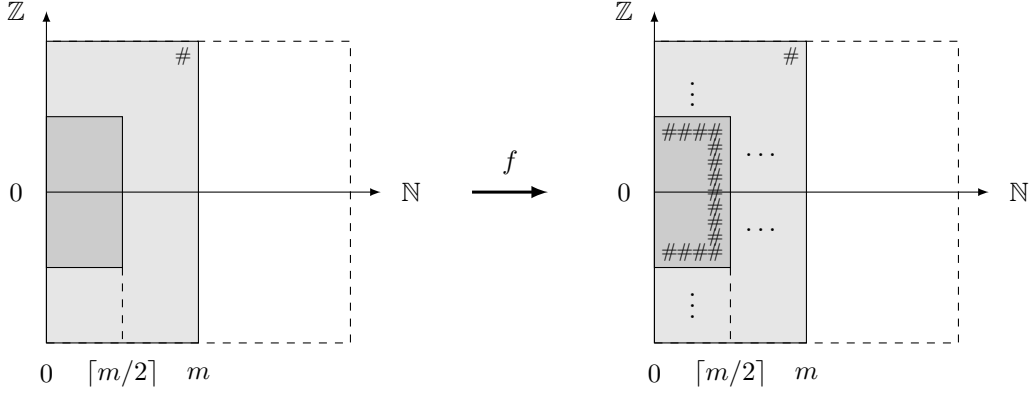
\begin{figure}[ht]
\[\begin{tikzpicture}[scale=0.2]
\begin{scope}
\fill[gray!20] (0,-10) rectangle (10,10);
\draw (0,-10) rectangle (10,10);
\fill[gray!40] (0,-5) rectangle (5,5);
\draw (0,-5) rectangle (5,5);
\draw[-latex] (0,-10) -- (0,12); 
\draw[-latex] (0,0) -- (22,0); 
\node at (-2,12) {$\Z$};
\node at (24,0) {$\N$};
\node at (-2,0) {$0$};
\node at (0,-12) {$0$};
\draw[dashed] (0,-10) -- (20,-10) -- (20,10) -- (0,10);
\draw[dashed] (5,-5) -- (5,-10); 
\node at (5,-12) {$\lceil m/2 \rceil$};
\node at (10,-12) {$m$};
\node[scale=0.75] at (9,9) {$\#$}; 
\draw[-latex,line width = 0.4mm] (28,0) -- (33,0);
\node at (30.5,2) {$f$};
\end{scope}

\begin{scope}[xshift=40cm]
\fill[gray!20] (0,-10) rectangle (10,10);
\draw (0,-10) rectangle (10,10);
\fill[gray!40] (0,-5) rectangle (5,5);
\draw (0,-5) rectangle (5,5);
\draw[-latex] (0,-10) -- (0,12); 
\draw[-latex] (0,0) -- (22,0); 
\node at (-2,12) {$\Z$};
\node at (24,0) {$\N$};
\node at (-2,0) {$0$};
\node at (0,-12) {$0$};
\draw[dashed] (0,-10) -- (20,-10) -- (20,10) -- (0,10);
\draw[dashed] (5,-5) -- (5,-10); 
\node at (5,-12) {$\lceil m/2 \rceil$};
\node at (10,-12) {$m$};
\node[scale=0.75] at (9,9) {$\#$}; 
\node[scale=0.75] at (4,4) {$\#$};
\node[scale=0.75] at (3,4) {$\#$}; 
\node[scale=0.75] at (2,4) {$\#$};
\node[scale=0.75] at (1,4) {$\#$}; 
\node[scale=0.75] at (4,3) {$\#$};  
\node[scale=0.75] at (4,2) {$\#$};  
\node[scale=0.75] at (4,1) {$\#$};  
\node[scale=0.75] at (4,0) {$\#$};  
\node[scale=0.75] at (4,-1) {$\#$};  
\node[scale=0.75] at (4,-2) {$\#$};  
\node[scale=0.75] at (4,-3) {$\#$};  
\node[scale=0.75] at (4,-4) {$\#$};  
\node[scale=0.75] at (3,-4) {$\#$}; 
\node[scale=0.75] at (2,-4) {$\#$};
\node[scale=0.75] at (1,-4) {$\#$};
\node at (2.5,7) {$\vdots$};
\node at (2.5,-7) {$\vdots$};
\node at (7,2.5) {$\hdots$};
\node at (7,-2.5) {$\hdots$};
\end{scope}
\end{tikzpicture}\]
\caption{\label{fig.diez} Illustration of the propagation of the 
error symbol $\#$ under the action of $f$.}
\end{figure}

\item \textbf{The entropy of $(\{0,1,\#\}^{\N \times \Z},f)$ is greater 
than the entropy of $(K,\sigma^{\textbf{e}^2})$:}

Since $K$ is stable under the action of $f$, 
which acts as $\sigma^{\textbf{e}^2}$ on 
this subset, the entropy of $(\{0,1,\#\}^{\N \times \Z},f)$ is greater than the entropy of the dynamical system 
$(K,\sigma^{\vec{e}^2})$, which is $h$.

\item \textbf{Upper bound on the minimal cardinality of a sub-cover of $\bigvee_{k \le n} f^{-k} (\mathcal{U}_l)$:}

This number is equal to the number of possible 
finite sequences 
$(x_{|\Delta_l},..., f^n(x)_{|\Delta_l})$, where 
$x \in \{0,1,\#\}^{\N \times \Z}$. In order 
to count these sequences, we distinguish two 
possibilities, as follows. For all $l,n$, 
we denote $M_{l,n,t}$ the number 
of patterns on $\llbracket 0, l \rrbracket \times \llbracket -l-n,l\rrbracket$ on alphabet $\{0,1\}$ that do not contain any pattern in $\{p_1,...,p_{t}\}$.

\begin{itemize}
\item \textbf{The pattern $f^n(x)_{|\Delta_l}$ 
does not contain any symbol $\#$:} this 
implies that $f^{n-1}(x)_{|\Delta_{2l}}$ 
does not contain any $\#$ or a pattern 
in $\{p_1,...,p_l\}$. Repeating this argument $n-1$ times, 
we get that $x_{|\Delta_{2^n l}}$ does not 
contain the symbol $\#$ or any 
pattern in $\{p_1,..., p_{2^{n-1} l}\}$. Moreover, since $f^{n-1}(x)_{|\Delta_{2l}}$ 
does not contain any $\#$ or a pattern 
in $\{p_1,...,p_l\}$, we have $f^{n}(x)_{\Delta_l} = f^{n-1} (x)_{\Delta_l - \vec{e}^2}$. Repeating this, 
we get 
$x_{|\Delta_l - n \vec{e}^2}= f^n (x)_{|\Delta_l}$. As a consequence, the number of sequences 
$(x_{|\Delta_l},..., f^n(x)_{|\Delta_l})$ 
that stand under this case is smaller than 
$M_{l,n,2^{n-1}l}$.

\item \textbf{The pattern $f^n(x)_{|\Delta_l}$ 
does contain a symbol $\#$:} this 
implies that $f^{n-1} (x)_{|\Delta_{2l}}$ 
contains else the symbol $\#$ or one 
of the patterns in $\{p_1,...,p_{l}\}$. In any 
of these cases, $f^{n-2} (x)_{|\Delta_{4l}}$ contains else $\#$ or a pattern in $\{p_1,...,p_{2l}\}$ and no $\#$. In the second case, 
this comes from the fact that $f^{n-2} (x)_{|\Delta_{2l}-\vec{e}^2} = f^{n-1}(x)_{|\Delta_{2l}}$. Repeating this argument, 
we obtain that $x_{|\Delta_{2^n l}}$ else 
contains $\#$ or a pattern in $\{p_1,...,p_{2^{n-1} l}\}$ and no $\#$.
Let us consider $m_0$ the maximal integer $m \ge -1$ such that $x_{|\Delta_m}$ does not contain $\#$ or a pattern 
in $\{p_1,...,p_{\lceil \frac{m}{2}\rceil}\}$.
If $m \ge 0$, $f(x)_{|\Delta_{\lceil \frac{m_0}{2}\rceil}}$ does not contain any forbidden 
pattern in $\{p_1,...,p_{\lceil \frac{m}{2}\rceil}\}$ or any symbol $\#$, and for all $\vec{u}
\notin \Delta_{\lceil \frac{m_0}{2}\rceil}$, 
$f(x)_{\vec{u}} = \#$.
We distinguish two possibilities: 
\begin{enumerate}
\item If $m_0 \le l$, then the sequence 
$(x_{|\Delta_l},..., f^n(x)_{|\Delta_l})$ 
is determined by $x_{|\Delta_l}$. 
The number of sequences in this case 
is bounded by $M_{l,0,0}$.
\item Else, $m_0 > l$, and we denote $(l_k)_{k \in \llbracket 0,n \rrbracket }$ the sequence 
such that $l_0 = m_0$ and for all $k \in \llbracket 0,n-1 \rrbracket$, $l_{k+1}= \lceil 
\frac{l_k}{2}
\rceil$. It is direct that for all $k \in \llbracket 1,n \rrbracket$, 
$f^k (x)_{|\Delta_{l_k}}$ does not contain 
any pattern in $\{p_1,...,p_{l_k}\}$ or 
any symbol $\#$ and for all $\textbf{u} \notin 
\Delta_{l_k}$, $f^k (x)_{\textbf{u}}=\#$.

We denote also $k_0$ the minimal $k \in \llbracket 1,n-1\rrbracket$ such that $l_k \le l$ 
(it exists since $f^n(x)_{|\Delta_l})$ contains 
the symbol $\#$).
The sequence $(x_{|\Delta_l},..., f^n(x)_{|\Delta_l})$ is determined by 
the restriction of $x$ to $\llbracket 0, l 
\rrbracket \times \llbracket -l-k_0,l\rrbracket$.
This pattern can be extended into a
pattern on $\Delta_{m_0}$ which 
does not contain a pattern in $\{p_1,..., p_{\lceil \frac{m_0}{2} \rceil}\}$. 
Moreover, since $l_{k_0} \ge m_0 / 2^{k_0}$ 
and $l_{k_0} \le l$, $m_0$ is bounded from 
above by $ 2^{k_0} \cdot l$. In particular 
the pattern $x_{|\llbracket 0, l 
\rrbracket \times \llbracket -l-k_0,l\rrbracket}$ can be extended into a pattern 
on $\Delta_{m_0}$ that does not 
contain a pattern in $\{p_1,..., p_{2^{k_0-1}\cdot l}\}$. In particular, $x_{|\llbracket 0, l 
\rrbracket \times \llbracket -l-k_0,l\rrbracket}$
does not contain any of these patterns. 
As a consequence, the number of sequences 
that stand in this case is bounded 
by $\sum_{k=1}^{n-1} M_{l,k,2^{k-1}l}$. \end{enumerate}
\end{itemize}

From this we deduce that
\begin{equation}\label{eq.1}
N_n (\{0,1,\#\}^{\N \times \Z},f,\mathcal{U}_l)\le M_{l,0,0} + \sum_{k=1}^n M_{l,k,2^{k-1}l}.
\end{equation}

\item \textbf{The entropy of $(\{0,1,\#\}^{\N \times \Z},f)$ is smaller
than the entropy of $(K,\sigma^{\textbf{e}^2})$:}

By equation (\ref{eq.1}) above,  we have that 
\[\inf_n \frac{\log_2 (N_n (f,\{0,1\}^{\N \times \Z},\mathcal{U}_l))}{n} = \lim_n \frac{\log_2 (N_n (f,\{0,1\}^{\N \times \Z},\mathcal{U}_l))}{n}\]
is smaller than:
\begin{align*}\liminf_n \frac{1}{n} \log \left(M_{l,0,0} + \sum_{k=1}^n M_{l,k,2^{k-1} l}\right) & = \liminf_n \frac{1}{n} \log \left(\sum_{k=k_1}^n M_{l,k,2^{k-1} l}\right)\\
& \le  \liminf_n \frac{1}{n} \log \left(\sum_{k=k_1}^n M_{l,k,2^{k_1-1} l}\right)\\
& \le \liminf_n \frac{1}{n} \log \left((n-k_1+1) \cdot M_{l,n,2^{k_1-1} l}\right)\\
& = \liminf_n \frac{1}{n} \log \left( M_{l,n,2^{k_1-1} l}\right)\\
& = \inf_n  \frac{1}{n} \log \left( M_{l,n,2^{k_1-1} l}\right)
\end{align*}
where $k_1 \ge 1$ is an arbitrary integer.
In the second inequality, we used the 
fact that the map $k \mapsto M_{l,k,2^{k_1-1} l}$ 
is non-decreasing. For the last 
equality, we used the fact that the sequence 
$(\log(M_{l,n,2^{k_1-1}l})/n)_n$ converges 
towards the entropy of $(K_{k_1},\sigma^{\vec{e}^2})$, where $K_{k_1}$ is the set of elements 
of $\{0,1\}^{\N \times \Z}$ where 
no pattern of $\{p_1,...,p_{k_1}\}$ appears, 
and that the entropy is also equal to the 
infimum of this sequence. We deduce 
that 

\[h(\{0,1,\#\}^{\N \times \Z},f,\mathcal{U}_l) \le \inf_n \frac{1}{n} \log (M_{l,n,2^{k_1-1}l}).\]
Since this inequality is verified for arbitrary 
$k_1 \ge 1$, this implies that 
\begin{align*}h(\{0,1,\#\}^{\N \times \Z},f,\mathcal{U}_l) & \le \inf_{k_1} \inf_n \frac{1}{n} \log (M_{l,n,2^{k_1-1}l})\\
& \le \inf_n \inf_{k_1} \frac{1}{n} \log (M_{l,n,2^{k_1-1}l}) \\
& = \inf_n \frac{1}{n} \log \left( N_{n+2l} (K,\sigma^{\vec{e}^2},\mathcal{U}_l)\right)\\
& = \lim_n \frac{1}{n} \log \left( N_{n+2l} (K,\sigma^{\vec{e}^2},\mathcal{U}_l)\right)\\
& = \lim_n \frac{n+2l}{n} \frac{1}{n+2l} \log \left( N_{n+2l} (K,\sigma^{\vec{e}^2},\mathcal{U}_l)\right)\\
\end{align*}

This yields that 
\[h(\{0,1,\#\}^{\N \times \Z},f,\mathcal{U}_l) \le h(K,\sigma^{\vec{e}^2},\mathcal{U}_l).\]

By taking the supremum over $l$ of these inequalities, since $(\mathcal{U}_l)$ is 
a generating sequence of open covers,
\[h(\{0,1,\#\}^{\N \times \Z},f) \le h(K,\sigma^{\vec{e}^2}).\]

\item \textbf{Recoding the system $(\{0,1,\#\}^{\N \times \Z},f)$ into a system $(\{0,1,\#\}^{\N},g)$:}

Let us consider $\phi$ a computable invertible 
map $\N \times \Z \rightarrow \N$. For instance, 
one can construct $\phi$ such that for all $n,k \ge 0$, $\phi (n,k) = 2^{2n}-1+2^{2n+1} 
k$ and  $\phi (n,-k-1) = 2^{2n+1}-1+2^{2n+1} k$.
Consider also the function 
$\psi : \{0,1,\#\} ^{\N \times \Z} \rightarrow 
\{0,1,\#\}^{\N}$ defined for all $u$, $n \ge 0$, and $k \in \Z$, by $(\psi(u))_{n,k} = u_{\phi(n,k)}$. $\psi$ and its inverse are then clearly computable.  Since $f$ is also computable, it follows that the function $g= \psi \circ f \circ \psi^{-1} $ 
is computable. Since $f$ and $g$ are conjugated, 
$h(\{0,1,\#\}^{\N},g) = h(\{0,1,\#\}^{\N \times \Z},f)$. But this entropy is equal to 
$h(K,\sigma^{\vec{e}^2})=h$ by the second 
and fourth points above. 
\end{enumerate}
\end{proof}

\subsection{\label{sec.surjectivity}
Surjective maps over $\{0,1\}^{N}$ are entropy-complete}

The functions constructed in the proof of 
Proposition~\ref{prop.cantor.real} 
are clearly not onto. In this Section, after some preliminaries on one-dimensional subshifts and related notions, we 
prove that the statement of Proposition~\ref{prop.cantor.real} stands under the additional 
constraint of surjectivity. Moreover, the constructed systems act on the canonical Cantor space $\{0,1\}^{\N}$.


Let us denote $\sigma : \{0,1\}^{\mathbb{N}} 
\rightarrow  \{0,1\}^{\mathbb{N}} $ 
the \textbf{shift} map, defined such that for all $x \in \{0,1\}^{\mathbb{N}}$, and $j \in \mathbb{N}$,
$(\sigma (x))_j= x_{j+1}$.
We say that a word $w\in\{0,1\}^{n}$ 
\textbf{appears} in a sequence $x \in \{0,1\}^{\mathbb{N}}$ 
when there exists some $k \in \mathbb{N}$ 
such that $x_{|\llbracket k , k+n-1\rrbracket} = w$. Note that over $X= \{0,1\}^{\mathbb{N}}$, 
the sets we call $\mathcal{U}_I$ simply correspond to the cylinders
$\{x \in \{0,1\}^{\N}: x_{|\llbracket k,k+n-1 \llbracket}=w\}$, 
where $k,n$ are integers and $w$ is a word of length $n$.

\begin{definition}
A \textbf{subshift} of the cantor 
set $\{0,1\}^{\mathbb{N}} $ is 
a closed subset $Z \subset \{0,1\}^{\mathbb{N}} $ 
such that $\sigma(Z) \subset Z$.
\end{definition}

In the literature, a subshift $Z$ of $\{0,1\}^{\mathbb{N}}$ 
is called \emph{effective} when there is an algorithm 
that enumerates a set $\mathcal{F}$  of words 
such that the configurations of $Z$ 
are defined by forbidding these words to 
appear on them: 

\[Z = \left\{x \in \{0,1\}^{\mathbb{N}}: \forall k,n \ge 0, x_{|\llbracket k , k+n-1\rrbracket} \notin \mathcal{F}\right\}.\]

Thus, the systems $(Z,\sigma_{|Z})$ where $Z$ is 
an effective subshift, are computable dynamical systems according to Definition \ref{computable-dynamics}. We call entropy of the subshift $Z$ the 
entropy of the dynamical system $(Z,\sigma_{|Z})$.
This entropy is known to be equal to 
\[h(Z) = \lim_n \frac{\log_2 (N_n (Z))}{n},\]
where $N_n (Z)=|\mathcal{L}_n (Z)|$, 
and $\mathcal{L}_n (Z)$ 
is the set of words of length $n$ that appear in a configuration of $Z$.
Let us denote $\mathcal{L} (Z) = 
\bigcup_n \mathcal{L}_n (Z)$. This set 
is called the \textbf{language} of $Z$.
In particular, when $h(Z)>0$, $Z$ cannot be a finite set (since if it 
was finite, $(N_n(Z))_n$ would be a bounded sequence). 

\begin{remark}Since the shift map is expansive, we obtain in virtue of Corollary \ref{expansive} that the entropies of effective subshifts are $\Pi_{1}$-computable numbers.
\end{remark}

\begin{definition}
A subshift $Z$ of $\{0,1\}^{\mathbb{N}}$ 
is said to be \textbf{decidable} when there 
exists an algorithm that taking as input 
a finite word on $\{0,1\}$ decides if it 
is in $\mathcal{L}(Z)$ or not.
\end{definition}

In particular, if a subshift $Z$  
is decidable, it is effective.

\begin{definition}
A subshift $Z$ of $\{0,1\}^{\mathbb{N}}$ 
is said to be \textbf{mixing} when
for any pair $(w,w')$ of words with respective lengths $m,n$ that appear in configurations of $Z$, there exists a configuration $z \in Z$ and two integers 
$k,k'$ such that $\llbracket k, k+m-1\rrbracket$ and $\llbracket k', k'+n-1\rrbracket$ are disjoint and  such that 
$z_{|\llbracket k, k+m-1\rrbracket} = w$ and $z_{|\llbracket k', k'+n-1\rrbracket} = w'$.
\end{definition}

\begin{lemma}
\label{lemma.mixing.decidable}
Let $Z$ be a mixing 
subshift of $\{0,1\}^{\mathbb{N}}$. 
Then $\sigma_{|Z}$ is surjective.
\end{lemma}
  
  \begin{proof}
  Indeed, it is sufficient to see that 
  any cylinder of $Z$, which 
  is the non empty intersection of $Z$ with 
  a cylinder of $\{0,1\}^{\mathbb{N}}$, 
  contains an image of $\sigma_{|Z}$. 
  Let us consider such a cylinder and 
  denote $u$ be a length $n$ word on $\{0,1\}$
  such that this cylinder is 
  $\{ x \in Z : x_{|\llbracket 0, n-1 \rrbracket} = u\}$. Also consider another length $m$ word $v$ 
  in the language of $Z$. Since $Z$ is mixing, there exists a configuration $z$ of $Z$ 
  and $k > m$ such that 
  $z_{|k+\llbracket 0,n-1\rrbracket} = u$ 
  and $z_{\llbracket 0,m-1\rrbracket}=v$. 
  Since $Z$ is a subshift, the configuration 
  $\sigma^{k} (x)$ is in the cylinder 
  $\{ x \in Z : x_{|\llbracket 0, n \rrbracket} = u\}$, and is the image of $\sigma^{k-1} (x)$.
  \end{proof}

We will need the following theorem. It realizes $\Pi_{1}$-computable numbers within the class of mixing decidable subshifts.   
  \begin{theorem}
\label{theorem.hellouin}
There exists an algorithm which, given as input
the code $n$ of an algorithm which enumerates 
a non-increasing sequence of non-negative rational numbers $(r_k)$, computes the code of 
an algorithm which decides the language of 
a decidable and mixing subshift $Z$ of $\{0,1\}^{\mathbb{N}}$ whose entropy is $\inf_k r_k$.
\end{theorem}

\begin{proof}
This follows from the proof of Theorem 3.7 in~\cite{hellouin}.
\end{proof}


We now show how to ``encode'' decidable mixing subshifts onto computable maps of the Cantor set.  
Let us fix some encoding of programs by integers.
For all integer $n$, if it is associated to 
a program that decides the language of a mixing 
subshift, this subshift is denoted $Z_n$.
Such an integer, when $h(Z_n) >0$ and $Z_n$ 
is mixing, is called \textbf{admissible}.

\begin{lemma}\label{recoding}
There exists a sequence of bi-computable homeomorphisms $\kappa_n : Z_n \rightarrow \{0,1\}^{\N}$, $n$ admissible. This sequence is \textit{uniformly computable} in the sense that there is an algorithm which, provided with an admissible $n$, computes $\kappa_{n}$ and its inverse. 
\end{lemma} 

\begin{proof} We note that when $Z$ is decidable, a function $\kappa: Z \rightarrow \{0,1\}^{\mathbb{N}}$ 
is computable if and only if there exists a non decreasing computable function $\phi : \mathbb{N} \rightarrow \mathbb{N}$ together with an algorithm which, provided with the $\phi(k)$ first symbols of a sequence $z$ in $Z$, computes the $k$ first symbols of the sequence $\kappa(z)$. We thus prove the lemma by showing the following: i) there exists an algorithm 
which given as input an admissible integer $n$ 
and a word in $\mathcal{L} (Z_n)$, outputs 
a word $\Psi(n,w)$ in $\{0,1\}^{*}$ satisfying the properties below, and ii) there exists a second algorithm which given as input an 
admissible integer $n$ and some other integer $k$,
outputs an integer $\phi(n,k)$ such that: 

\begin{itemize}
\item For all admissible $n$, $z \in Z_n$ and 
$k \ge 0$, $(\Psi(n,z_{|\llbracket 0,k\rrbracket}))$ is a growing sequence 
of words; 
\item denoting $\Psi(n,z)$ 
the limit sequence, for all $z \in Z_n$, 
$\kappa_n (z) = \Psi(n,z)$.
\item For all $z \in Z_n$ and $k$, the length of $\Psi(n,z_{|\llbracket 0, \phi(n,k)\rrbracket})$ 
is $\ge k$.
\end{itemize}

The details are as follows. 

\begin{enumerate}
\item \textbf{Definition of $\Psi$:}
Let us consider some admissible integer $n$, 
and $w \in \mathcal{L} (Z_n)$.
The word $\Psi(n,w)$ is the result of the 
following algorithm: 
\begin{enumerate}
\item Initialize an empty writting tape,
\item For all $k \in \llbracket 0, m-1\rrbracket$, 
where $m$ 
is the length of $w$ (step $k$), 
consider the word 
$w_{\llbracket 0,k \rrbracket}$ and the word 
 obtain from this one by flipping the last letter: 
 
 \begin{itemize}
 \item 

 If this second word is in $\mathcal{L}(Z_n)$ (use 
 the algorithm deciding the language of $Z_n$ for 
 this purpose), write the value of $w_k$ at the 
 end of the word on the writting tape. 
 \item 
Else, don't write anything on the tape.
 
 \end{itemize}
 \end{enumerate}
 
 \item \textbf{For all $n$ admissible and $z \in Z_n$, the limit sequence $\Psi(n,z)$ is well defined:}
 
 Indeed, by definition of the algorithm presented 
 in the first point, when some word $w \in \mathcal{L}_n (Z)$ 
 is prefix of another one $w' \in \mathcal{L}_{n+k} (Z)$, 
 the sequence of steps executed 
 by the algorithm defining $\Psi(n,w)$ is 
 prefix of the sequence 
 of steps executed by the algorithm defining 
 $\Psi(n+k,w')$. The complementary 
 steps only (potentially) add symbols.
 In other words, for all $z \in Z_n$, 
 $(\Psi(n,z_{\llbracket 0, k \rrbracket}))_k$ 
 is a growing sequence of words. The limit 
 $\Psi(n,z)$ is a priori finite, but is proved 
 to be infinite below.
 
 \item \textbf{The $\kappa_{n}$ are injective:}

Let us consider two different 
elements of $Z_n$, denoted $z$ and $z'$ and 
denote by $k_0$ the smallest integer $k$ such that 
$z_k \neq z'_k$. 
This means that any symbol written by the 
algorithm defining $\kappa_n$ before step $k_0$ is the same for $z$ and $z'$. 
However, when considering the integer $k_0$, 
since $z'_{|\llbracket 0,k_0\rrbracket}$ 
is obtained by flipping the last letter of 
$z_{|\llbracket 0,k_0\rrbracket}$ and these two 
words are in the language of $Z_n$,
a letter is written on the tape during this step, 
which is different for $z$ and $z'$. 
As a consequence, $z$ and $z'$ have a different 
image through $\kappa_n$.
 
 \item \textbf{The $\kappa_{n}$ are onto:}
 
 In order to prove that for all $z \in Z_n$, 
 $\kappa_n (z)$ is an infinite sequence, it is sufficient to see 
 that any word written on the 
 tape during the execution of the algorithm 
 defining $\kappa_n (z)$ is extended in the next 
 steps. Assuming this, $\kappa_n$ has 
 images in $\{0,1\}^{\N}$ and since any extension by a symbol means that the extension by the other symbol is possible,  $\kappa_n : Z_n \rightarrow \{0,1\}^{\N}$ is surjective.
 In order to prove this, it is enough to 
see that for any word $w$ in 
 the language of $Z_n$, there exists some $k$ such 
 that $w$ can be extended into two different length $k$ words in $\mathcal{L}(Z_n)$. Let 
 us fix such a word $w$.
 Since $h(Z_n)>0$, the subshift $Z_n$ is an infinite set, and thus the two letters $0,1$ are in $\mathcal{L}(Z_n)$. Since $Z_n$ is mixing, there 
 exists some $k \in \mathbb{Z}$ and two configurations $z^0$ and $z^1$ in $Z_n$ 
 such that the restriction of $z_0$ and $z_1$ 
 on $\llbracket 0, |w| \rrbracket $ is $w$, 
 and $z^0_k = 0$ and $z^1_k = 1$. The 
 restrictions of $z^0$ and $z^1$ on 
 $\llbracket 0, k \rrbracket$ are thus different.
 
 \item \textbf{Uniform computability 
 of the minimal mixing function of $Z_n$, 
 $n$ admissible:}
 
 There is an algorithm which given an admissible integer $n$ and an integer $k$ 
 outputs an integer $\psi(n,k)$
 such that for all $w,w' \in \mathcal{L}_k (Z_n)$, 
 there exists a configuration $z \in Z_n$ 
 such that $z_{|\llbracket 0,k-1 \rrbracket} = w$ 
 and $z_{|\llbracket k+\psi(n,k),\psi(n,k)+2k-1 \rrbracket} = w$. Indeed, such 
 an algorithm checks if this property is true 
 for successive integers, using the 
 algorithm deciding the language of $Z_n$, 
 stops when the property is true, and outputs 
 the current integer.

 \item \textbf{Proof of the third point and computability of $\kappa_{n}$:}
 For all $n$ admissible, let us denote $F: k \mapsto k + \psi(n,k)$, 
 and $\phi : (n,k) \mapsto F^k (1)$. For all $w \in \mathcal{L}_k (Z_n)$ and $a \in \{0,1\}$, there exists a configuration 
 $z \in Z_n$ such that $z_{|\llbracket 0, k-1 \rrbracket } = w$ and $z_{k+\psi(n,k)}=a$.
 Hence for any sequence $z \in Z_n$, once the algorithm 
 defining $\kappa_n (z)$ scanned the $F^k(1)$th first symbols of the sequence, it has determined 
 $k$ symbols of its image $f(x)$, which finishes the proof. 

\end{enumerate}
\end{proof}


We are now ready to state and prove the main result of the section.  

\begin{theorem}
\label{theorem.realization.surjective}
A non-negative real number 
is entropy of an effective dynamical 
system $(\{0,1\}^{\N},f)$ with $f$ surjective 
if and only if it is $\Sigma_2$-computable.
\end{theorem}

\noindent\textit{Outline of proof.} The main idea is this: given $0<h = \sup_n \ h_n$ with $h_{n}$ a computable sequence of $\Pi_{1}$-computable numbers, one can  produce, uniformly in $n$, a mixing decidable subshift whose entropy is $h_{n}$. By applying Lemma \ref{recoding}, we can 
construct a system $f_n$ of the Cantor set with entropy $h_{n}$. Since the subshift is mixing, the obtained function is surjective. One can then construct a surjective function of the Cantor set whose entropy is $h$ 
by decomposing the Cantor set into 
a countable union of smaller and smaller copies of it, 
and defining the global dynamics $f$ to be like $f_n$ on the $n$th copy.  Now to the details. 

\begin{proof}
The case $h=0$ is ruled out by seeing 
this number as entropy of the identity function, 
which is onto.
Let $h > 0$ be a $\Sigma_2$-computable number, 
and denote $(r_{m,n})_{m,n}$ a computable double sequence 
of rational numbers such that 
$h = \sup_n \inf_m r_{m,n}$.
For all $n$, we denote $h_n = \inf_m r_{m,n}$.
Without loss of generality, one 
can assume that $(h_n)_n$ is 
non-decreasing by replacing $(r_{m,n})_{m,n}$
$r'_{m,n} = \max_{k \le n} r_{m,k}$, which is 
a computable sequence, and 
$(h'_n)_n$ is increasing, where 
$h'_n = \inf_m r'_{m,n}$, since 
$h'_n = \max_{k \le n} h_k$. 
For the same reason, $h = \sup_n h'_n$.

Let us denote $n_0$ some integer 
such that for all $n \ge n_0$, $h_n >0$.

\begin{enumerate} 

\item \textbf{Realization of the numbers 
$h_n$:}

One can, given an integer $n \ge n_0$, compute 
the code $\tau(n)$ of an algorithm 
which decides the language of a decidable and mixing subshift $Z_{\tau(n)}$ 
of $\{0,1\}^{\mathbb{N}}$ whose entropy is $h_n$: 
this is a consequence of Theorem~\ref{theorem.hellouin}. 
Let us denote $f_n : \{0,1\}^{\mathbb{N}} \rightarrow \{0,1\}^{\mathbb{N}}$ the map
\[f_n \equiv \kappa_{\tau(n)} \circ \sigma \circ \kappa_{\tau(n)}^{-1}.\]

\item \textbf{Surjectivity of
$f_n$:}

For all $n$, since $Z_{\tau(n)}$ is mixing and decidable, the map $\sigma : Z_{\tau(n)} \rightarrow Z_{\tau(n)}$ is surjective (Lemma~\ref{lemma.mixing.decidable}). 
As a consequence, $f_n$ is surjective.

\item \textbf{Realization of $h$:}

Let us define $f : \{0,1\}^{\mathbb{N}} \rightarrow \{0,1\}^{\mathbb{N}}$ by $f(0^k 1 x) = 0^k 1 f_k(x)$ 
for all $k \in \N$. 
We have \[h(f)=\sup_n h_n = h.\] Indeed, 
for all $l$, 
the entropy of $f$ relatively to the cover 
$\mathcal{U}_l$ is the supremum 
of the numbers $h_k$, $k \le l$, 
where $\mathcal{U}_l$ is the open cover 
of all the cylinders corresponding to length 
$l$ words on $\{0,1\}$.

Since the topological entropy of $f$ is the 
supremum over $l$ 
of the entropies of $f$ relatively to 
$\mathcal{U}_l$, this entropy is 
$\sup_l \sup_{k \le l} h_k = \sup_l h_l$.

\item \textbf{Surjectivity and computability 
of $f$:}
Since for all $n$, $f_n$ is surjective, $f$ 
is also surjective. Moreover, $f$ is computable: 
indeed, in order to compute the $n$ first bits 
of $f(x)$ given $x$, one first localise 
the first $1$ in the sequence $(x_0 , ... , x_{n-1})$. If there is no $1$, the $n$ first 
bits of $f(x)$ are all $0$. Else, let us denote 
$k$ the first such that $x_k = 1$. 
One has to compute the $\phi(n,n-k)$ 
first bits of $(x_{k+1+j})_{j \ge 0}$ in order 
to compute the $n-k$ first bits of $f(x)$ that 
are not determined. Since 
$n \mapsto \max_{k \le n} \phi(k,n-k)$. 
is computable, this
yields the computability of $f$.
\end{enumerate}
\end{proof}

\section{Computable systems over the unit interval}
\label{section.continuous.dynamics}

In this section we present the proof of Theorems C and D. We start by characterizing the possible entropies of the whole class of computable maps over $[0,1]$.

\subsection{\label{sec.characterization.interval}The class of all computable maps}

Our goal here is to prove the following. 

\begin{theorem} 
\label{th.characterization.interval}
The numbers that are the entropy of a computable dynamical system 
$([0,1],f)$ are exactly the $\Sigma_1$-computable non-negative numbers.
\end{theorem}

The proof is split into Propositions \ref{prop.obstruction.int} and \ref{prop.realisation.int} below. We will start by proving the computational \emph{obstruction}, i.e., that no computable system over $[0,1]$ can have topological entropy whose arithmetical complexity is beyond $\Sigma_{1}$.  For this, we need first to briefly recall some properties and results about topological entropy that will be used in the proof of Proposition  \ref{prop.obstruction.int}, which will mostly rely on the notion of \emph{horseshoe} ( see Definition \ref{def.horseshoe}). 



\begin{definition} \label{def.monotonic}
We say that a continuous map $f : 
[0,1] \rightarrow [0,1]$ 
is piecewise monotone when there exists $k$ and some numbers 
$x_1=0 < ... < x_k=1$ such that on each $[x_i , x_{i+1}]$ 
the restriction of $f$ on this interval is monotone.
When $f$ is strictly monotone on each interval 
$[x_i , x_{i+1}]$,we say that 
$f$ is strictly piecewise monotone.
\end{definition}

\begin{lemma} \label{lem.monotonic.iterate}
Let $f : [0,1] \rightarrow [0,1]$ be a 
continuous strictly piecewise 
monotone function. For all $n$, $f^n$ is 
strictly piecewise monotone.
\end{lemma}

\begin{proof}
We use a recurrence argument. Since this property is true for $n=1$, 
this is sufficient to prove that if 
$f$ and $g$ are two strictly piecewise 
monotone functions, then $f \circ g$ is
 also strictly piecewise monotone.
Let $f,g$ be two such functions.
There exists $k$ and some numbers 
$x_1=0 < ... < x_k=1$ such that on each $[x_i , x_{i+1}]$ 
the restriction of $f$ on this interval is 
increasing or decreasing, and 
there exists $k'$ and some numbers 
$x'_1=0 < ... < x'_{k'}=1$ such that on each 
$[x'_i , x'_{i+1}]$ 
the restriction of $g$ on this interval 
is strictly monotone. 
Let us denote $(x'')_{i \le k''}$ 
the finite increasing 
sequence such that 
\[\{x''_i\} = \{x_i : i \le k\} \bigcup 
\{x'_i : i \le k'\}.\]
On each of these intervals, both $f$ and 
$g$ are strictly monotone, and as a consequence, 
$f \circ g$ is also strictly monotone.
\end{proof}

The following notion is a powerful tool to analyse dynamical systems on the interval, and in particular the topological entropy of these systems. 

\begin{definition} \label{def.horseshoe} 
Let $f$ be a continuous map of the interval $[0,1]$. A $(p,n)$ horseshoe of $f$ 
is a finite collection of compact disjoints intervals $J_1 , J_2 , ... , J_p$ such that for all $i$, $f^n (J_i)$ contains 
a neighbourhood of $\cup_j J_j$. The horseshoe is 
called monotone when on every $J_j$, $f^n$ is strictly monotone.
\end{definition}

The proof of the following theorem can be found in~\cite{Ruette15}. This theorem will be used for the proof of the obstruction.

\begin{theorem}[\cite{MisiurewiczMonotone}] 
\label{th.misiurewicz1}
Let $f$ be a continuous map of the interval $[0,1]$. Then \[h (f) = 
\sup_{(p,n) \in \Delta} \frac{\log_2 (p)}{n},\]
where $(p,n) \in \Delta \Leftrightarrow f$ 
admits a $(p,n)$-horseshoe. When 
$f$ is a piecewise monotone map, 
\[h (f) = 
\sup_{(p,n) \in \Delta'} \frac{\log_2 (p)}{n},\]
where $(p,n) \in \Delta' \Leftrightarrow f$ 
admits a $(p,q)$ monotone horseshoe.
The number $\log_2 (p)/n$ is called the entropy 
the $(p,n)$-horseshoes.
\end{theorem}

\begin{definition}
Let $f : [0,1] \rightarrow [0,1]$ be a continuous map. 
Its \define{variation} 
is the following number 
\[V(f)=\sup \sum_{k=1}^n |f(x_k) - f(x_{k+1})|,\]
where the finite sequence $(x_k)_{k}$ is such that 
$x_k = 0 < ... < x_{n+1} = 1$.
\end{definition}

The following lemma is stated as 
Corollary 15.2.14 in~\cite{Katok}: 

\begin{lemma}[\cite{Katok}] 
\label{lem.variation}
Let $f : [0,1] \rightarrow [0,1]$ be a 
continuous map. Then we have the following 
equality: 
\[h(f)=\lim_n \frac{\log_2 (V(f^n))}{n}.\]
\end{lemma}


We are now ready to prove the obstruction part: 

\begin{proposition} 
\label{prop.obstruction.int} 
For all $f : [0,1] \rightarrow [0,1]$, $h (f)$ 
is $\Sigma_1$-computable.
\end{proposition}

\begin{remark}
By convention, $+\infty$ is considered as a $\Sigma_1$-computable number.
\end{remark}

\begin{proof}

For $I$ and $J$ two intervals, we write $I<J$ when for all $x \in I$ 
and $y \in J$, $x<y$. Let $I_{1},\dots,I_{p}$ be a finite collection of rational intervals such that $I_1 < ... < I_p$ (in particular, the compact intervals $\overline{I_i}$ are disjoint). Our goal is to show that there is an algorithm to semi-decide whether $\overline{I_{1}},\dots,\overline{I_{p}}$ is a $(n,p)$ horseshoe of $f$, for some $n$.  That is, whether for all $i\leq p$ there exists $\epsilon>0$ such that $\mathcal{I}^{\epsilon}\subset f^{n}(\overline{I_{i}})$, 
where $$\mathcal{I}=\bigcup_{i=1}^{p} \overline{I_{i}}$$ and $A^{\epsilon}$ stands for an $\epsilon$-neighborhood of a set $A$.  Note that since $f$ is computable, $f^{n}(\overline{I_{i}})$ is simply an interval with computable endpoints, say $a$ and $b$, which may or may not belong to $f^{n}(\overline{I_{i}})$. Since $\mathcal{I}$ is compact, we have that 
$$
\mathcal{I}^{\epsilon}\subset f^{n}(\overline{I_{i}}) \text{ for some } \epsilon>0 \iff \mathcal{I}\subset ]a,b[. 
$$
Moreover,  the set $\mathcal{I}$ is also an effective closed set and therefore by Proposition \ref{prop.equivalence.comp.closed.subset} we can semi-decide whether it is contained in a given finite union of open rational intervals. Since $a$ and $b$ are computable, there is a recursively enumerable $L\subset \N$ such that 
$$
]a,b[=\bigcup_{l\in L}]l_{l},r_{l}[,
$$
from which the desired semi-decidability property for horseshoes follows. We now can use this property to algorithmically enumerate all rational horseshoes of $f$. Since the entropies $\frac{\log_2(p)}{n}$  of these horseshoes are computable numbers, their supremum is a $\Sigma_{1}$-computable number.  To finish the proof, it suffices to note that an arbitrary $(p,n)$ horseshoe can be turned into a rational $(p,n)$ horseshoe by a sufficiently small modification, and therefore this last supremum equals the entropy of $f$ by Theorem~\ref{th.misiurewicz1}. 
\end{proof}

The following proposition gives a realization of all the $\Sigma_1$-computable numbers, thus completing the proof of Theorem \ref{th.characterization.interval}.

\begin{proposition} \label{prop.realisation.int}
Let $h\geq 0$ be a $\Sigma_1$-computable real number. 
There exists a computable function 
$f : [0,1] \rightarrow [0,1]$ 
such that the entropy of $([0,1],f)$ is $h$. 
\end{proposition}

\begin{proof}We show how to realize all computable entropies first. 
\begin{enumerate}
\item \textbf{Realization of computable numbers:} 
our goal here is to prove that given $h$ a positive computable real number, we can compute a function $f_h$ of the interval $[0,1]$ with $f_h(0)=0$, $f_h(1)=1$ and such that the topological entropy of 
$f_h$ is $h$. Since for any map $f$ the equality $h(f^n)=nh(f)$ holds, it is sufficient to prove the result for $h \in [0,1]$ only. The case $h=0$ is trivial, so assume $h>0$ and denote $s=2^h$, which is also a computable real number. Now consider the function $f_h$ defined as follows: 
\begin{itemize}
\item when $x \le \frac{1+s}{4s}$, then 
\[f_h (x)=sx,\]
\item when $\frac{1+s}{4s} \le x \le \frac{3s-1}{4s}$, then 
\[f_h (x)= \frac{1+s}{4}-s\left(x-\frac{1+s}{4s}\right),\]
\item and when $x \ge \frac{3s-1}{4s}$, 
\[f_h (x)= s\left(x-\frac{3s-1}{4s}\right)-\frac{3s-1}{4}.\]
\end{itemize}

See an illustration in Figure~\ref{fig.tent.map}.

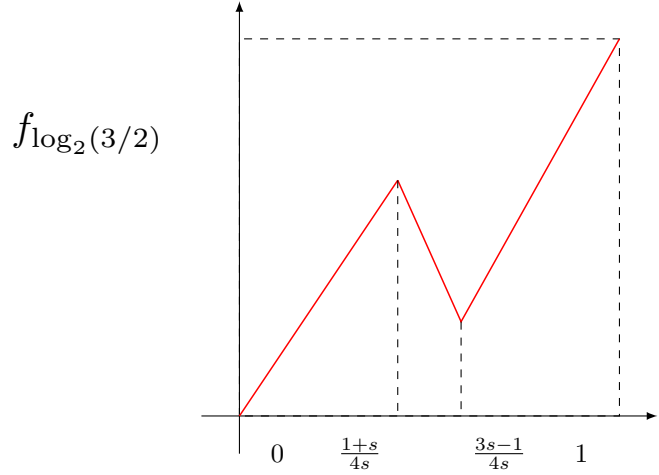
\begin{figure}[ht]
\[\begin{tikzpicture}[scale=0.5]
\draw[-latex] (0,-1) -- (0,11);
\draw[-latex] (-1,0) -- (11,0);
\draw[dashed] (0,0) rectangle (10,10);
\draw[color=red,line width=0.2mm] (0,0) -- (4.166,6.25);
\draw[color=red,line width=0.2mm] (4.166,6.25) -- (5.833,2.5) -- (10,10);
\draw[dashed] (4.166,6.25) -- (4.166,0);
\draw[dashed] (5.833,2.5) -- (5.833,0);
\node at (6.833,-1) {$\frac{3s-1}{4s}$};
\node at (3.166,-1) {$\frac{1+s}{4s}$};
\node at (1,-1) {$0$};
\node at (9,-1) {$1$};
\node[scale=1.5] at (-4,7.5) {$f_{\log_2(3/2)}$};
\end{tikzpicture}\]
\caption{\label{fig.tent.map} Illustration of the map $f_h$ when $s=3/2$.}
\end{figure}

Since on each of the intervals defining $f_h$ the function has slope $s$, we see that, 
for all $n$, the function $f_h^n$ has slope $s^n$. We claim that the variation of $f_h^n$ is $s^n$. Indeed, 
from Lemma~\ref{lem.monotonic.iterate}, there exists a sequence 
$x_1 =0 < ... < x_k = 1$ such that on each of the intervals 
$[x_i,x_{i+1}]$, this function is increasing or decreasing, and 
the slope of the function is $s^n$. Considering a finite sequence 
$y_1 =0 < ... < y_{k'} =1$, 
\[\sum_{j=1}^{k'-1} |f_h^n (y_{j+1})-f_h^n (y_{j+1})| \le 
\sum_{j=1}^{k''-1} |f_h^n ({\tilde{y}}_{j+1})-f_h^n ({\tilde{y}}_{j+1})|,\]
where the sequence $\tilde{y}$ is defined 
to be the increasing sequence such that 
\[\{\tilde{y}_i : i \le k''\} = 
\{x_i : i \le k\} \bigcup \{y_i : 
i \le k'\}.\]
On each of the intervals 
$[{\tilde{y}}_{j+1},{\tilde{y}}_{j+1}]$, 
the slope of the function is $s^n$, hence 
\[|f_h^n ({\tilde{y}}_{j+1})-f_h^n ({\tilde{y}}_{j+1})|= s^n |
{\tilde{y}}_{j+1}- {\tilde{y}}_{j+1}|= s^n ({\tilde{y}}_{j+1}-{\tilde{y}}_{j+1}).\]
This implies that the variation of $f_h^n$ is 
\[V(f_h^n) = \sup s^n \left( \sum_{j} \left( {\tilde{y}}_{j+1}-{\tilde{y}}_{j+1} 
\right) \right) = s^n.\]

From Lemma~\ref{lem.variation}, we deduce that 
the entropy of $f_h$ is 
\[h(f_h) = \lim_n \frac{\log_2 (s^n)}{n} = h.\]

It remains to show the computability of the function $f_h$. But this easily follows from the computability of $s$ and the explicit form of $f$ in terms of $s$.  

\item \textbf{Realization of 
$\Sigma_1$-computable numbers:}
Let now $h$ be a positive $\Sigma_1$-computable real number, 
and $(h_n)$ an increasing 
computable sequence of rational numbers 
such that $h_n \rightarrow h$. 

Let $f_h$ be the map defined by, for all $x, n \ge 0$, 
if $x \in \left[ \sum_{k=1}^{n} \frac{1}{2^k} , 
\sum_{k=1}^{n+1} \frac{1}{2^k}\right]$, 
\[f_h (x) = \sum_{k=1}^{n} \frac{1}{2^k} + 
\frac{1}{2^n} f_{h_n}.\]
See an illustration in Figure~\ref{fig.tent.map2}.

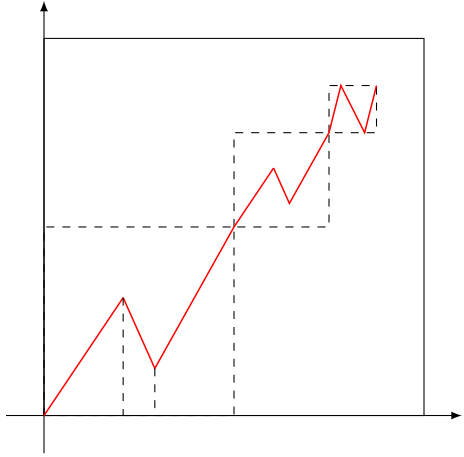
\begin{figure}[ht]
\[\begin{tikzpicture}[scale=0.5]
\draw[-latex] (0,-1) -- (0,11);
\draw[-latex] (-1,0) -- (11,0);
\draw (0,0) rectangle (10,10);
\draw[dashed] (5,5) rectangle (7.5,7.5);
\draw[dashed] (7.5,7.5) rectangle (8.75,8.75);
\draw[dashed] (0,0) rectangle (5,5);
\draw[color=red,line width=0.2mm] (0,0) -- (4.166/2,6.25/2);
\draw[color=red,line width=0.2mm] (4.166/2,6.25/2) -- (5.833/2,2.5/2) -- (5,5);
\draw[color=red,line width=0.2mm] (5,5) -- (5+4.166/4,5+
6.25/4);
\draw[color=red,line width=0.2mm] (5+4.166/4,5+6.25/4) -- (5+5.833/4,5+2.5/4) -- (7.5,7.5);
\draw[color=red,line width=0.2mm] (7.5,7.5) -- (7.5+1.25/4,
8.75) -- (7.5+3.75/4,7.5) -- (8.75,8.75);
\draw[dashed] (4.166/2,6.25/2) -- (4.166/2,0);
\draw[dashed] (5.833/2,2.5/2) -- (5.833/2,0);
\end{tikzpicture}\]
\caption{\label{fig.tent.map2} Illustration 
of an example of map $f_h$, when $h_1 = h_2 = 
\log_2(3/2)$ and $h_3 = 1$.}
\end{figure}

The variation of $f_h^n$ is 
\[V(f_h^n) = \sum_{k=1}^{+\infty} \frac{1}{2^k}
h_k ^n,\]
hence we have the following inequalities, for all 
$p$, since $(h_k)_k$ 
is increasing: 
\[\sum_{k=1}^{p} \frac{1}{2^k} h_k^n 
+ \frac{1}{2^{p-1}} h_{p}^n 
\le V(f_h^n) \le \sup_{k} h_k^n = h^n\]
This implies that the entropy of $f_h$ verifies 
\[h_p \le h(f_h) \le h,\]
and $h(f_h) = h$.

The function $f_h$ is computable. 
Indeed, it is sufficient to see 
that the inequality $f_h (q) < r$
and $f_h (q) > r$ are semi-decidable, 
where $q,r$ are rational numbers. 

An algorithm which allows to 
semi-decide the first one 
(it is similar for the second one) 
works as follows:

\begin{enumerate}
\item it first checks if one of the rational 
numbers is $0$ or $1$. In this case, 
it is trivial to decide the inequality.
\item if none of these rational is $0$ or $1$, 
then the algorithm computes some $n$ 
such that 
\[\sum_{k=1}^n \frac{1}{2^k} \le 
q \le \sum_{k=1}^{n+1} \frac{1}{2^k}.\]
\item then it computes more and more precise 
approximations of the number $h_n$ - 
this is possible since $(h_n)_n$ is uniformly 
computable - checks if the number 
\[2^n (q-\sum_{k=1}^n \frac{1}{2^k})\] 
is in one of the intervals defining 
the function $f_{h_n}$.
If this is the case, the algorithm computes 
the image, and checks if this image is $<r$. 
If this is the case, the algorithm stops.
\end{enumerate}

This algorithm stops when the inequality 
is verified: this comes 
from the fact that $(h_n)_n$ is 
computable.

\end{enumerate}
\end{proof}



We end this section by noting that from the proof of Proposition \ref{prop.realisation.int} we obtain as a by product the following characterization of computable numbers. For a real number $s$, 
we denote $\mathcal{FS}_{s}$ the class of piecewise linear continuous maps  $f : [0,1] \rightarrow [0,1]$ such that the slope of $f$ is $\pm s$ on each of the intervals where it is linear, and let $\mathcal{FS}=\cup_{s} \mathcal{FS}_{s} $. 

\begin{proposition} \label{th.comp}
A real number $h>0$ is computable if and only if it is the entropy of a computable 
map in $\mathcal{FS}$. 
\end{proposition}

\begin{proof} If $f$ is 
in $\mathcal{FS}_{s}$ for some real number 
$s$ and is computable,
then $s$ is a computable number.  
Indeed, if $x_1 = 0 < ... x_k = 1$ is 
such that on each $[x_i,x_{i+1}]$, the restriction 
of $f$ is linear, then considering 
a rational number $r$ such that $x_1 < r < x_2$, there is 
an algorithm which on input $n$ outputs some rational numbers $r_n$ and $r'_n$ such that 
$|r_n - f(r)| \le r . 2^{-n-1}$ and $|r'_n - f(0)| \le r . 2^{-n-1}$. Hence 
\[|(r_n - r'_n) - (f(r)-f(0))| \le |r_n - f(r)| + |r'_n - f(0)| \le r. 2^{-n},\]
which means that 
\[\left| \frac{r_n-r'_n}{r} - s \right| \le 2^{-n}.\]
Since $h(f)=\log(s)$, then the entropy of $f$ is computable. 
Reciprocally, since for any function $f$ 
an any integer $n$ the equality $h(f^n)=nh(f)$ holds, it is enough to only prove the result for numbers in $[0,1]$. But this follows directly from the construction presented in the first point of the proof of Proposition~\ref{prop.realisation.int}.
\end{proof}


\subsection{The class of quadratic maps}
\label{quadfam}

The result we will prove in this section is in fact stronger than Theorem D.  Indeed, we will show that the topological entropy $h(f_{r})$ of the logistic map $f_{r}$ is computable \emph{as a function} of $r$. The fact that all computable numbers can be realized within this class will then just follow as a corollary. We start by recalling the material related to the real quadratic family (also known as the logistic family) that will be required for the proof. 

\begin{definition}Let $\mathbb{I}=[0,1]$. The quadratic family is the collection of maps $f_r: \mathbb{I} \to \mathbb{I}$ defined by 
$$f_{r}(x)=rx(1-x); \quad \text{where } r\in[0,4] \,\,\text{ is \emph{the parameter}}.$$ 
\end{definition}

The behaviour of this family as the parameter $r$ varies has been extensively studied, see for instance \cite{Kurka}. Note that the equation $f_r(x)=x$ has two solutions: $p_0=0$ and $p_r=(r-1)/r$. These are the fixed points of $f_r$. The local dynamics near a fixed point is called \textit{stable} when all nearby orbits converge to the fixed point, and \textit{unstable} when nearby orbits escape away from the fixed point. The following is a useful criterion: 

\begin{theorem}[\textbf{Differential criterion of stability}]\label{diffcriterionstability} Let $f_r'(x_0)$ be the derivative of $f_r$ evaluated at a point $x_0$. A fixed point $p$ of $f_r$ is:
\begin{enumerate}[label=(\alph*)]
\item An \textsf{attracting fixed point} if $\vert{f_r'(p)}\vert<1$.
\item An \textsf{unstable fixed point} if $\vert{f_r'(p)}\vert>1$.
\end{enumerate}
\end{theorem}

As the parameter $r$ varies the dynamical behaviour of the trajectories of $f_r$ changes. We summarize part of its behaviour as follows:
\begin{itemize}
\item For $0<r<1$, every trajectory converges to $p_0$, which is the only fixed point and, moreover, is an attracting fixed point.
\item For $1\leq r<3$, $p_0$ looses its stability and another fixed point $p_r$ appears, which behaves as an attracting fixed point.
\item For $3<r<\sqrt6+1$, $p_r$ looses stability and an attracting cycle of period $2$ appears around $p_r$ which attracts all points in $\mathbb{I}=[0,1]$.
\item For  $r$ beyond $\sqrt6+1$, each point in the attracting cycle of period $2$ looses its stability and is splitted into two points, which together, form another attracting cycle of period $4$. This process continues as $r$ increases. Let $r_i$ be the values of $r$ for which an attracting cycle of period $2^i$ appears. We denote as $r_\infty=\lim_ir_i$ the parameter corresponding to the limit of the bifurcation process just described.
\item At $r=r_\infty$ the bifurcations stop and the attracting orbit becomes infinite in a subset of $\mathbb{I}$.
\item For $r_\infty <r\leq4$ the system's behavior seems to alternate between order and disorder, until  $r=4$, where it becomes fully chaotic according to  Devaney's definition (see \cite{devaney}).
\end{itemize}

In Figure \ref{bif} we present a plot of the the bifurcation diagram where the previous observations are represented in a more comprehensive way.
\begin{figure}[H]
\centering
\includegraphics[width=\textwidth, height=10cm]{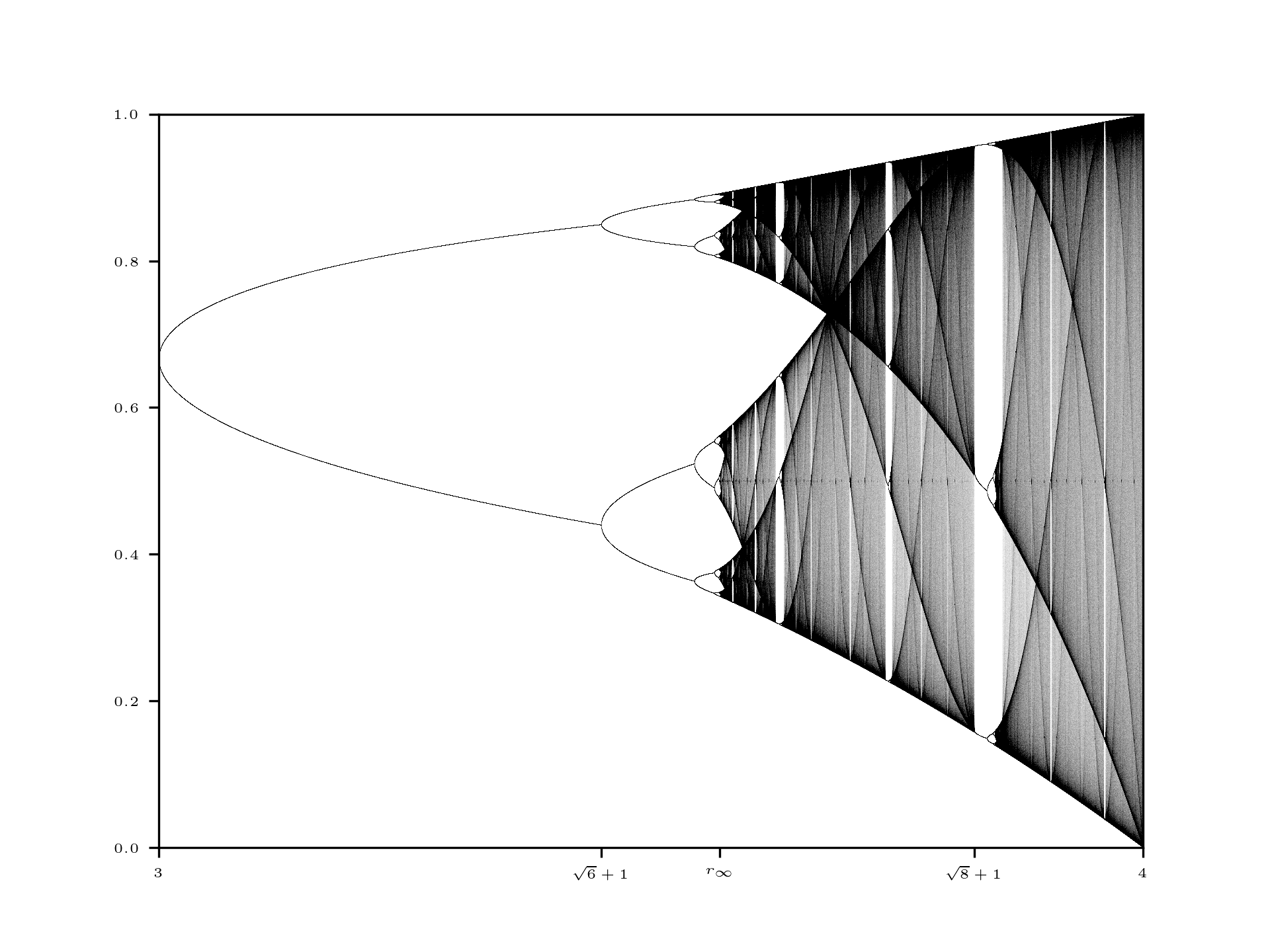}
\caption{Bifurcation Diagram: it plots the  \textit{attractor} of the map $f_r$ as a function of the parameter $r\in[3,4]$.}\label{bif}
\end{figure}

Note the big ``window" at $\sqrt8+1$. This is one of the regions known as \textit{hyperbolic components}, where ordered behavior seems to prevail. The set of  hyperbolic parameters  are,  in fact, dense in the space of parameters (see \cite{swiatek1992hyperbolicity, GraczSwiat}). At the beginning of this big window, it may seem that a periodic cycle of period 3 attracts \textit{all} the possible trajectories. However, the diagram here is misleading, since it is easy to show that if this was the case then the topological entropy would be zero, contradicting the following theorem due to Douady.

From now on, we will write $h(r)$ for the topological entropy of the system $(\mathbb{I},f_r)$. 

\begin{theorem}[Theorem 5.3 in \cite{douady}]\label{douady} The following properties hold for the quadratic family: 
\begin{enumerate}[label=(\alph*)]
\item The function $h:[0,4]\to[0,\ln2]$ is continuous and weakly increasing.
\item $h(r)>0$ if and only if $r>r_\infty$.
\item $h(r)=h(r')$ if and only if $r$ and $r'$ are tuned by the same $r_0$, with $h(r_0)>0$
\end{enumerate}
\end{theorem}

The main theorem of this section is:

\begin{theorem}\label{entropyquadraticfamily} $h:[0,4]\to[0,\ln2]$ is  a computable function. 
\end{theorem}

Before presenting its proof, we note that Theorem D follows from Theorem \ref{entropyquadraticfamily}. via the following corollary. 

\begin{corollary}[Theorem D] A real number in $[0,\ln2]$ is computable if and only if it is the topological entropy of a computable logistic map. 
\end{corollary} 
\begin{proof}
By Theorem \ref{entropyquadraticfamily} it is clear that computable maps have computable topological entropy.  Let's show the converse. Note that since $h$ is nondecreasing and onto, given a computable $y\in[0,\ln2]$, the set $h^{-1}(\{y\})$ is either a singleton or an interval. The latter case is trivial since any interval must contain computable points. In case it is a singleton $\{r\}$, since $h$ is computable by Theorem \ref{entropyquadraticfamily}, $\{r\}$ must be an effective set. By Proposition \ref{prop.equivalence.comp.closed.subset}, we can then semi-decide whether $\{r\}\subset (a,b)$ uniformly in $a$ and $b$, from which computability of $r$ follows. 
\end{proof}

The proof of Theorem \ref{entropyquadraticfamily} will follow from Propositions \ref{seqr}, \ref{hunifcomp} and \ref{sequence} below, but the main strategy is based on three facts
\begin{itemize}
\item In every hyperbolic component there is an attracting cycle which contains the critical point $c=1/2$.
\item The topological entropy is locally constant inside every hyperbolic component
\item The hyperbolic components are dense in the space of parameters (see \cite{GraczSwiat}).
\end{itemize}
We use the following known result as well.

\begin{theorem}[Theorem 4.11 \cite{keriko}]\label{keriko}All roots of an analytical computable function on $[0,1]$ are computable.
\end{theorem}

Recall that a subshift of finite type with pruning of words of length two is known as a Markov subshift and the computability of its topological entropy, given its alphabet $A$ and the set of forbidden words $\mathcal{F}$, is a known result that can be found, for instance, in Theorem 6.1 of \cite{spandl2007computing}.

With all these ingredients we start with the following result, which provides a list of parameters where the attracting cycle contains the critical point $c$.

\begin{proposition}\label{seqr}
There exists an algorithm that enumerates a sequence $(r_i,p_i)_{i}\subseteq\mathbb{R}\times\mathbb{N}$  such that the critical point $c$, under iterations by $f_{r_i}$, has a attracting periodic orbit of period $p$. 
\end{proposition}
\begin{proof}
First note that the equation $f_{r}(c)$, iterated $p-1$ times, is a polynomial in $r$ of degree $2^p-1$, we define the polynomial $P_p(r)=f_{r}^{p}(c)-c$ of degree $2^p-1$. Recall from Theorem \ref{keriko} that all the roots of analytic functions, such as polynomials, are computable. Let us denote by $\mathcal{K}$ the algorithm that on input $P_p$ outputs all the roots of $P_p$ -- we consider a single algorithm $\mathcal{K}$ since the roots are uniformly computable for these parameters.\\
We describe an algorithm $\mathcal{A}$ which enumerates the required sequence. $\mathcal{A}$ starts a parallel simulation of $\mathcal{K}$ with inputs $P_p$ for $p=1,2,3,\ldots$. Every time one of the parallel simulations of $\mathcal{K}$ halts, for some $p$, $\mathcal{A}$ checks whether $r\in(0,4)$, whenever this is satisfied $\mathcal{A}$ label $p$ with $i$ and outputs the pair $(r_i,p_i)$, where $r_i$ is a root of $P_{p_i}$. This procedure provides the desired sequence.
\end{proof}

\textbf{Observation:} in the proof of Proposition \ref{seqr}, $r_i$ is a computable real number ``provided" by algorithm $\mathcal{K}$ -- what $\mathcal{K}$ really outputs is, for each $r_i$, a sequence $$\vert{q_j- r_i}\vert\leq 2^{-j},\quad q_j\in\mathbb{Q}.$$
If the critical point $c$ is part of a finite attracting cycle, a Markov partition can be induced over the unit interval $\mathbb{I}$ and the corresponding Markov subshift can be defined. The algorithmic procedure for the construction of the required subshift is described in the proof  Proposition  \ref{hunifcomp} below, which asserts that the topological entropies of the systems corresponding to the parameters obtained by the means of Proposition \ref{seqr}, are uniformly computable.

\begin{proposition}\label{hunifcomp}
The sequence $\left(h(f_{r_i})\right)_i$ is uniformly computable.
\end{proposition}
\begin{proof}
Given the sequence $(r_i,p_i)_{i}$ from Proposition \ref{seqr}, we describe an algorithm $\mathcal{A}_{sh}$ which uniformly computes the sequence $\left(h(f_{r_i})\right)_i$.

The algorithm $\mathcal{A}_{sh}$ starts, for each input $(r_i,p_i)$, by computing the set  $\mathcal{O}=\cbra{f_{r_i}^{n}(c):0<n\leq p_i}$ which is the orbit of $c$. Then $\mathcal{A}_{sh}$ orders the elements in $\mathcal{O}$ with respect to the relation ``$<$, \textit{less than}" in an ordered set $\mathcal{P}$, and adds the numbers $0$ and $1$ as the smallest and biggest elements of $\mathcal{P}$. Then, the elements of $\mathcal{P}$:
$$0=x_0<x_1<\ldots<x_k=c<\ldots<x_{p_i}<x_{p_i+1}=1.$$
form a partition of $\mathbb{I}$, with atoms
$$A_{\mathcal{P}}=\cbra{[x_j,x_{j+1}): 0\leq j\leq p_i}.$$
Later, $\mathcal{A}_{sh}$ creates the set $A=\cbra{a_j:0\leq j\leq p_i}$ that labels the atoms in $\mathcal{P}$, and computes the set $A^2=\cbra{a_ia_j: a_i, a_j\in A}$, by concatenating the letters of $A$ in all possible combinations. Next $\mathcal{A}_{sh}$ computes the set $\mathcal{F}$ as follows: for the first atom $a_0$, $\mathcal{A}_{sh}$ computes $f_{r_i}(x_1)=x_m$, which is true for some $0\leq m\leq p_i$. Then $f_{r_i}(a_0)=\cup_{n=0}^m\cbra{a_n}$. With this computation done, $\mathcal{A}_{sh}$ creates the collection of words $\mathcal{F}_0=\cbra{a_1a_n\in A^2:n>m}$. 
Suppose that $x_k=c$. For $a_j$, $0<j< k$, $\mathcal{A}_{sh}$ computes $f_{r_i}(x_j)=x_s$ and $f_{r_i}(x_{j+1})=x_t$, which are true for some $0\leq s,t\leq p_i$. Therefore $f_{r_i}(a_j)=\cup_{n=s}^t\cbra{a_n}$. This holds since $\mathcal{P}$ is a Markov partition and the image, under $f_{r_i}$, of every atom in the partition is a union of atoms in the partition. Once the previous computations are completed, $\mathcal{A}_{sh}$ computes the collection of words $\mathcal{F}_j=\cbra{a_ja_n\in A^2: 0<n<s,\ t<n<1}$ for each $0<j\leq k$.

Similarly, for $k<j\leq p_i$, $\mathcal{A}_{sh}$ computes $f_{r_i}(x_j)=x_u$ and $f_{r_i}(x_{j+1})=x_v$, which are true for some $0\leq u,v\leq p_i$. Since the slope of $f_{r_i}$, for $x>c$, is negative $v<u$. Then the image of $a_j$ is an union of atoms  -- $f_{r_i}(a_j)=\cup_{n=v}^u\cbra{a_n}$. Now $\mathcal{A}_{sh}$ proceeds to compute the sets $\mathcal{F}_j=\cbra{a_ja_n\in T^2: 0\leq n<v, u<n<1}$. Note that the sets $\left(\mathcal{F}_j\right)_{0\leq j\leq p_i}$ are disjoint and represents, symbolically, all the trajectories forbidden in terms of the atoms that are visited by each point under iterations of $f_{r_i}$.

Finally, with all the previous computations done (which is a finite number of computations), $\mathcal{A}_{sh}$ computes $$\mathcal{F}=\cup_{j=0}^{c_i+p_i+1}\cbra{\mathcal{F}_j}.$$

Note that $\mathcal{F}$ defines a Markov subshift $\Sigma_\mathcal{F}$, since every element in $\mathcal{F}$ has length $2$ and is a forbidden word in the subshift. 
From Theorem 6.1 of \cite{spandl2007computing} the computability follows. This computation can be done in a parallel simulation for each input $(r_i,p_i)$ given by the algorithm $\mathcal{A}_{sh}$ from Proposition \ref{seqr}. Since this is the same for any sequence outputted by $\mathcal{A}_{sh}$, the result  follows.
\end{proof}

Now we have the topological entropy for a collection of parameters $r_i$. But this collection is not dense in $[0,4]$ yet (each $r_i$ is at the center of a hyperbolic component). Since a dense sequence is needed for our purposes we will use this collection to find a dense set of parameters for which we know how to compute the entropy.

\begin{proposition}\label{sequence}
For each pair $(r_i,p_i)_{i}$ there exists an algorithm that enumerates a sequence $(d_k,h(r_i))_k$, such that $f_{d_k}^{p_i}$ has an attracting cycle of period $p_i$ and topological entropy $h(r_i)$.
\end{proposition}
\begin{proof}
We describe an algorithm $\mathcal{A}_s$ that enumerates the desired sequence. As in Proposition \ref{seqr}, we invoke an algorithm $\mathcal{K}_s$ which computes uniformly the roots $x_n$ of the polynomial $f_{r_i}^{p_i}(x)-x$, for every $(r_i,p_i)$. Recall that rational numbers are enumerable, let $\mathcal{E}$ be the algorithm that enumerates the sequence $d_j$, with $d_j\in\mathbb{Q}\cap[0,4]$.

Upon input $(r_i,p_i)$, $\mathcal{A}_s$ simulates $\mathcal{E}$ and for every one of its outputs $d_j$, $\mathcal{A}_s$ starts, in parallel, a simulation of algorithm $\mathcal{K}_s$ with input $(d_j,p_i)$ for all $j\geq0$. For every output $x_j$ of $\mathcal{K}_s$, $\mathcal{A}_s$ computes $\vert{{f_{d_j}^{p_i}}'(x_j)}\vert$ and semi-decides the inclusion ${f_{d_j}^{p_i}}'(x_j)\in(-1,1)$. If this condition holds, we output $d_j$. By Theorem \ref{diffcriterionstability}, $f_{d_j}^{p_i}$ has an attracting cycle of period $p_i$ and therefore is in the same hyperbolic component as $f_{r_i}$, as desired. Moreover by Theorem \ref{douady}(c) we have that $h(r_i)=h(d_k)$ for all such $d_k$, as it was to be shown. 
\end{proof}

We are now equipped to prove that given any parameter $r$ as input, there is an algorithmic procedure which outputs an arbitrarily good approximation of $h(r)$ when provided with arbitrarily good approximations of $r$. The details are as follows: 

\begin{proof}[Proof of Theorem \ref{entropyquadraticfamily}]
We describe and algorithm $\mathcal{A}_h$ which computes the topological entropy $h(r)$ for any given parameter $r\in[0,4]$ at any desired precision $\epsilon$. Note that, since $r\in[0,4]$ is a real number, we assumed that better approximations of $r$ are provided whenever they are required.
On input $(r,\epsilon)$, $\mathcal{A}_h$ starts a parallel simulation of algorithm $\mathcal{A}_s$ from Proposition \ref{sequence}. For every output $(d_i,h(d_i))$ of $\mathcal{A}_s$, $\mathcal{A}_h$ checks in a parallel way with better and better approximations of $r$ whether $d_i<r$ or $d_i>r$. Whenever one of these conditions is verified, $\mathcal{A}_h$ stores $(d_i,h(d_i))$ in sequences $(d_u,h(d_u))_u$, if $d_i>r$, or $(d_l,h(d_l))_l$ if $d_i<r$. For both sequences $\mathcal{A}_h$ can relabel every term to ensure that $(d_u)_u$ is decreasing and $(d_l)_l$ is increasing.  As the sequences $(d_u,h(d_u))_u$ and $(d_l,h(d_l))_l$ start to be enumerated, $\mathcal{A}_h$ checks whether the condition
$$h(d_u)-h(d_l)<\epsilon-\frac\epsilon{10}$$
is satisfied, for some $u,l$. Once this is achieved, $\mathcal{A}_h$ halts and outputs the value $\left(h(d_u)+h(d_l)\right)/2$, which is the desired approximation of $h(r)$.
\end{proof}


%
%

\bibliographystyle{abbrv}
\bibliography{ref,biblio}
\end{document}